%% file: main.tex
\documentclass[11pt]{article}

\usepackage{amsthm}
\usepackage{hyperref}
\usepackage{amssymb}
\usepackage{latexsym}
\usepackage{amsmath}
\usepackage{amsfonts}
\usepackage{mathrsfs} 
\usepackage{version}
\usepackage[dvips]{graphicx}
\usepackage[LGR,T1]{fontenc}%
\usepackage{url}
\usepackage{enumerate}
\usepackage{color}
\usepackage{bm, bbm}
\usepackage{textgreek}

\usepackage{pgf,tikz}
\usetikzlibrary{arrows}
\usetikzlibrary{automata, chains}
\usetikzlibrary{positioning}

\tikzset{%
  every picture/.style={thick,shorten >=1pt,>=stealth'},
  state/.append style={%
    state without output,draw=blue!50,very thick,fill=blue!20,%
    inner sep=0pt,minimum size=4.5ex}
}

\numberwithin{equation}{section}
\theoremstyle{plain}
\newtheorem{thm}{Theorem}
\numberwithin{thm}{section}
\newtheorem{prop}{Proposition}
\numberwithin{prop}{section}
\newtheorem{lm}{Lemma}
\numberwithin{lm}{section}

\numberwithin{term}{section}

\numberwithin{claim}{section}
\newtheorem{cor}{Corollary} 
\numberwithin{cor}{section}

\numberwithin{problem}{section}

\newtheorem{assump}{Assumption}
\numberwithin{assump}{section}

\newtheorem{defn}{Definition} 
\numberwithin{defn}{section}

\newtheorem{question}{Question}

\numberwithin{conjecture}{section}

\numberwithin{condition}{section}
\theoremstyle{remark} 
\newtheorem{remark}{Remark}
\numberwithin{remark}{section}

\newcommand{\Proj}{\mathrm{Proj}}
\newcommand{\Dom}{\mathrm{Dom}}

\newcommand{\N}{\mathbb{N}}
\newcommand{\bD}{\mathbbm{D}}
\newcommand{\bK}{\mathbbm{K}}
\newcommand{\bQ}{\mathbb{Q}}

\newcommand{\Lip}{\mathrm{Lip}}

\newcommand{\Law}{\mathrm{Law}}

\newcommand{\ud}{\,\mathrm{d}}
\newcommand{\pa}{\mathrm{pa}}

\newcommand{\EE}{\mathbb{E}}

\newcommand{\PP}{\mathbb{P}}
\newcommand{\FF}{\mathbb{F}}
\newcommand{\RR}{\mathbb{R}}
\newcommand{\mcC}{\mathcal{C}}
\newcommand{\mcF}{\mathcal{F}}

\newcommand{\mcI}{\mathcal{I}}

\newcommand{\mcM}{\mathcal{M}}

\newcommand{\mcP}{\mathcal{P}}
\newcommand{\mcS}{\mathcal{S}}

\newcommand{\mc}[1]{\mathcal{#1}}

\newcommand{\tila}{\tilde{a}}
\newcommand{\tilX}{\widetilde{X}}
\newcommand{\tilY}{\widetilde{Y}}

\newcommand{\tilb}{\widetilde{b}}

\newcommand{\tilpartial}{\tilde{\partial}}

\numberwithin{equation}{section}

\title{Smoothness of Directed Chain Stochastic Differential Equations}
\author{ Tomoyuki Ichiba \thanks{Department of Statistics and Applied Probability, South Hall, University of California, Santa Barbara, CA 93106, USA (E-mail: \href{mailto:ichiba@pstat.ucsb.edu}{ichiba@pstat.ucsb.edu}). Research was supported in part by National Science Foundation NSF DMS-2008427.}
\and Ming Min \thanks{Department of Statistics and Applied Probability, South Hall, University of California, Santa Barbara, CA 93106, USA (E-mail: \href{mailto:m_min@pstat.ucsb.edu}{m\_min@pstat.ucsb.edu}).}
}

\begin{document}
\maketitle

\begin{abstract}
     We study the smoothness of the solution of the directed chain stochastic differential equations, where each process is affected by its neighborhood process in an infinite directed chain graph, introduced by Detering et al. (2020). Because of the auxiliary process in the chain-like structure, classic methods of Malliavin derivatives are not directly applicable. Namely, we cannot make a connection between the Malliavin derivative and the first order derivative of the state process. It turns out that the partial Malliavin derivatives can be used here to fix this problem. 
\end{abstract}

\section{Introduction} 
\input{intro}

\section{Preliminaries and Directed Chain SDEs}
\label{sec:directed_sde}
\input{directedSDE}

\section{Partial Malliavin Calculus}
\label{sec:partialMalliavin}
\input{malliavin}

\section{Smoothness of Densities}
\label{sec:smooth}
\input{density}

\bibliographystyle{acm}
\bibliography{citations}


\end{document}

%% file: intro.tex
The main objective of this paper is to study the existence and regularity of the densities of the directed chain stochastic differential equations: given a filtered probability space $(\Omega, \mcF, (\mcF_t)_{t\ge 0}, \PP)$, the directed chain McKean-Vlasov stochastic differential equation (or directed chain SDE for short) for a pair $(X_\cdot^\theta, \tilX_\cdot)$ of $N$-dimensional stochastic processes considered here is of the form
\begin{equation}
    \begin{split} \label{eq:1.1}
	& X_t^\theta = \theta + \int_0^t V_0(s, X_s^\theta, \Law(X_s^\theta), \tilX_{s}) \ud s + \sum_{i=1}^d \int_0^t V_i(s, X_s^\theta, \Law(X_s^\theta), \tilX_{s}) \ud B_s^i, 
	\end{split}
\end{equation}
for $t \ge 0$ with the distributional constraint 
$$
[ X_t^\theta, t \ge 0 ] := \Law(X_t^\theta, t \ge 0 ) = \Law(\tilX_t, t \ge 0 ) =: [\tilX_t, t \ge 0 ], 
$$
where $V_i$, $i = 1, \ldots , d$ are some smooth coefficients, $B_\cdot:=(B_\cdot^1, \dots, B_\cdot^d)$ is a standard $d$-dimensional Brownian motion independent of the initial state $X_0^\theta = \theta$ and of $\tilX_\cdot$, and $X_0^\theta$ is independent of $\tilX_0$. 
Throughout the paper, $[\xi]$ denotes the law of a generic random element $\xi$. Here each coefficient $V_i$ in \eqref{eq:1.1} depends on time $s$, the value $X_s^\theta$, its law $\Law(X_s^\theta) =: [X_s^\theta]$ and the other $\tilX_s$ of the pair for $s\ge 0$. The law $[X^\theta_\cdot]$ depends on the law $[\tilX_\cdot]$ through \eqref{eq:1.1} and they are the same marginal law. We show that the above directed chain SDE has a unique weak solution in section \ref{sec:directed_sde}.

This kind of directed chain structure was firstly proposed by \cite{DETERING20202519} in a simpler form. Schematically, it can be written as an infinite chain of stochastic equations for $(X_{1,\cdot}, X_{2,\cdot}, \ldots ) $:  
\begin{equation}
    \begin{split}
	\ud X_{1,t} &= b(t, X_{1,t}, F_{1,t}) \ud t  + \ud B_{1,t}, \label{def:sde_X} \\
	\ud {X}_{2,t} &= b(t, {X}_{2,t}, {F}_{2,t}) \ud t  + \ud {B}_{2,t}\\
	& \vdots  \\ 
	\ud {X}_{i,t} &= b(t, {X}_{i,t}, {F}_{i,t}) \ud t  + \ud {B}_{i,t}, \\
	& \vdots 
\end{split}
\end{equation}
where $F_{i,t} := u \delta_{X_{i+1,t}} + (1-u) \mu_{i,t}$ is the mixture distribution term of the measure-dependent drift coefficient $b$ with the marginal law $\mu_{i,t}:= \mathrm{Law}(X_{i,t})$ of $\,X_{i,t}\,$ for $\,t \ge 0 \,$, $\delta_{X_{i+1,t}}$ is the Dirac measure at ${X}_{i+1,t}$, a fixed constant $u\in [0,1]$ measures the common amount of dependency of $X_{i,\cdot}$ on its neighborhood value $X_{i+1,\cdot}$, and $B_{1,\cdot}, {B}_{2,\cdot, \ldots }$ are independent standard Brownian motions. 
We assume also that the initial value $X_{i,0}$ is independent of $B_{i,\cdot} $, and $\,X_{i+1,\cdot}\,$ and $\, {B}_{i,\cdot}\,$ are independent for $i = 1, 2, \ldots $.  In particular, the drift $b$ in \cite{DETERING20202519}  has the following form
$b(t, x, \mu) := \int_\RR \tilb(t, x, y) \mu (\!\ud y)$ 
with some Lipschitz continuous function $\widetilde{b}$. See also Figure \ref{fig:directedchain} in section \ref{sec:smooth}. 

The stochastic processes on infinite graphs including the directed chain structure have drawn many attentions recently.  Stochastic Differential Games on the directed chain have been studied in \cite{yichen1} and on the extended version of random directed networks in \cite{yichen2} as well as on the general random graph (e.g., \cite{lacker2020case}). \cite{lacker2021locally} discuss the Markov random field property over both finite and countably infinite graph with local interactions through the drift coefficients. Another related topic is the Graphon particle system. There are a sequence of works in Graphon Mean Field Games,  \cite{bayraktar2020graphon, caines2018graphon, caines2019graphon} just to name a few. \cite{bayraktar2021graphon} introduced the uniform-in-time exponential concentration bounds related to the graphon particle system and its finite particle approximations. Here, we are interested in the existence and smoothness of the density of directed chain SDE  \eqref{def:sde_random_coef}-\eqref{def:sde_random_coef_constrain}. It should be emphasized here that in this problem, we need notions of derivatives in the space of measures, which is used frequently in the theory of Mean Field Games. 

In most cases, Malliavin calculus is a foundation to analyze the smoothness of the density of stochastic differential equations. It has been widely used in investigating the density of diffusions 
\cite{Kusuoka1984ApplicationsOT}, \cite{KUSUOKA1984271}, \cite{Kusuoka1985ApplicationsOT} and then applied into many different scenarios. The authors in \cite{CASS20091416} use Malliavin calculus to derive smoothing properties of solution to stochastic differential equations with jumps. The smoothing properties of McKean-Vlasov SDEs have been studied in \cite{crisan2018smoothing}, which is closely related to our purpose. However, because of the appearance of the auxiliary process $\tilX$, the crucial step making connections between the Malliavin derivative and the first order derivative of the state process fails, please see Question \ref{Q:integral_by_parts} for the detail. {To our best knowledge, we did not find any works studying the smoothness property of such weak solutions of stochastic differential equations.}

For the purpose of resolving this problem and utilizing the  Malliavin derivatives, we should frozen the auxiliary process $\tilX$. This inspired us to consider another closely related, well-developed tool, partial Malliavin calculus. Partial Malliavin calculus is first introduced by \cite{Kusuoka1984partial} for the constant case, where the projections are taken on a fixed Hilbert subspace, and applied to prove some regularity results in Non-linear Filtering theory. 
Another work developing the partial Malliavin calculus is \cite{ikeda1985malliavin}, by which the authors were able to complete the proof of some results in \cite{malliavin1982certaines} on the long-time asymptotic of the stochastic oscillatory integrals. We mainly adopt the framework from a later work by Nualart and Zakai \cite{nualart1989partial}, where the projection is taken on a family of the subspace which is defined as the orthogonal complement to the subspaces generated by $\tilX$ in \eqref{eq:1.1}. {We remark that our method is potentially applied to analyze the smoothness property of weak solutions of stochastic differential equations in a general setting.}

This paper is structured as follows: In section \ref{sec:directed_sde}, we first introduce the differentiation in the space of measures and multi-index notation in section \ref{sec:Notations and Basic Setup}, and then prove the existence, uniqueness and some regularity results on the solutions of generalized directed chain SDEs in Propositions \ref{prop:sol of directed chain sde}-\ref{prop:regularity_Xtheta}. In section \ref{sec:partialMalliavin}, we prepare the notions of the partial Malliavin calculus and give the Kusuoka-Stroock process for the proof of our smoothness of densities, which will be stated in section \ref{sec:smooth}. Our proofs follows the idea in \cite{crisan2018smoothing}, where we first derive integral by parts formulae for the directed chain SDEs via the partial Malliavin derivatives, instead of the Malliavin derivatives, as in \cite{crisan2018smoothing}. The main result is stated in Theorem \ref{thm:smooth_density} with some applications in section \ref{sec:smooth}. 


%% file: directedSDE.tex
In this section, we first prepare some notations and the notion of differentiation in $\mc{P}_2$, where $\mc{P}_2$ is the space of all measures with finite second moments, and then establish the weak solutions of directed chain SDEs. 

\subsection{Notations and Basic Setup}
\label{sec:Notations and Basic Setup}
To be consistent with the reference \cite{crisan2018smoothing}, we use $[\xi]$ to denote the law of a random variable $\xi$. Rather than the directed chain SDE of the type given in \cite{DETERING20202519}, we consider the SDE in a more general setup, allowing the diffusion coefficients non-constant. Given a probability space $(\Omega, \mcF, \FF = (\mcF_t)_{t\ge 0}, \PP)$, the directed chain McKean-Vlasov SDE (or directed chain SDE for short) is of the form
\begin{align}
\label{def:sde_random_coef}
	& X_t^\theta = \theta + \int_0^t V_0(s, X_s^\theta, [X_s^\theta], \tilX_{s}) \ud s + \sum_{i=1}^d \int_0^t V_i(s, X_s^\theta, [X_s^\theta], \tilX_{s}) \ud B_s^i, \\
	\label{def:sde_random_coef_constrain}
	& \text{with the constraint }\quad [X_t^\theta, t \ge 0 ] = [\tilX_t, t \ge 0 ], 
\end{align}
where $B_s:=(B_s^1, \dots, B_s^d)$ is a standard $d$ dimensional Brownian motion and $\tilX_{s}\in L^2(\Omega\times [0,T], \RR^N)$ is an adapted random process independent of all the Brownian motions $B^i, i = 1,\dots,d$ and initial state $\theta$.

Moreover, we assume that $V_0, V_i: [0,T]\times\RR^N \times \mcP_2(\RR^N) \times \RR^N \to \RR^N$, where $\mcP_2(\RR^N)$ is the set of measures on $\RR^N$ with finite second moments. We equip $\mcP_2(\RR^N)$ with the $2$-Wasserstein metric, $W_2$. For a general metric space $(M, d)$, we define the $2$-Wasserstein metric on $\mcP_2(M)$ by
$$W_2(\mu, \nu) = \inf_{\Pi\in \mcP_{\mu, \nu}} \bigg( \int_{M\times M} d(x, y)^2 \Pi(\ud x, \ud y) \bigg)^{1/2},$$
where $\mcP_{\mu, \nu}$ denotes the class of measures on $M\times M$ with marginals $\mu$ and $\nu$. 

We denote $L^p$ norm on $(\Omega, \mcF, \PP)$ by $\|\cdot\|_p$, $p \ge 1$ and for every $t \ge 0$, we also introduce the space $\mcS_t^p$ of continuous $\FF$ adapted process $\varphi$ on $[0,t]$, satisfying 
$$\|\varphi\|_{\mcS_t^p} = \big(\EE \sup_{s\in[0,t]} |\varphi_s|^p\big)^{1/p}<\infty.$$

Let us introduce more notations in accordance with \cite{crisan2018smoothing}. We will write $\theta=\delta_x$ if the initial state of this SDE is a fixed real vector $x \in \mathbb R^N$. We use $\mcC^{k,k,k}_{b, \Lip}(\RR^+\times\RR^N\times \mcP_2(\RR^N) \times \RR^N ; \RR^N)$ for the class of functions that are $k$-times continuously differentiable with bounded Lipschitz derivatives in the the last three variables, where the notion of derivatives with respect to measure is adopted from P.-L. Lions' lecture notes at the \textit{Coll\`{e}ge de France}, recorded in a set of notes \cite{cardaliaguet2012notes}, very well exposed in \cite{carmona2015} and also adopted by \cite{crisan2018smoothing}. A precise definition for $\mcC^{k,k,k}_{b, \Lip}(\RR^+\times\RR^N\times \mcP_2(\RR^N) \times \RR^N ; \RR^N)$ will be given in Definition \ref{def:Ckkk}.

\paragraph{Differentiability in $\mcP_2(\RR^N)$.}
Lion's notion of differentiability with respect to measure of functions $U:\mcP_2(\RR^N)\to \RR$ is to define a lifted function $U'$ on the Hilbert space $L^2(\Omega'; \RR^N)$ over probability space $(\Omega', \mcF',\PP')$, where $\Omega'$ is a Polish space and $\PP'$ is an atomless measure, such that $U'(X') = U([X'])$ for $X' \in L^2(\Omega'; \RR^N)$ and $[X'] = [X]$. Thus, we are able to express the derivative of $U$ w.r.t. measure $\mu = [X]$ term as the Fr\'{e}chet derivative of $U'$ w.r.t. $X'$ whenever it exists, which can be written as an element of $L^2(\Omega'; \RR^N)$ by identifying $L^2(\Omega'; \RR^N)$ and its dual. This gradient in a direction $\gamma' \in L^2(\Omega'; \RR^N)$ is given by
$$\mc{D}U'(X')(\gamma') = \langle \mc{D}U'(X'),  \gamma' \rangle = \EE' \big[ \mc{D} U'(X') \cdot \gamma' \big],$$ 
where $\EE'$ is the expectation under $\PP'$. By \cite[Theorem 6.2]{cardaliaguet2012notes}, the distribution of this gradient depends only on the measure $\mu$, exists uniquely and can be written as 
$$\partial_\mu U(\mu, X'):=\mc{D}U'(X') = \xi(X') \in L^2(\Omega'; \RR^N).$$

This definition of the derivative with respect to measure can be extended to higher orders by thinking of $\partial_\mu U(\mu, \cdot): \mcP_2(\RR^N)\times\RR^N \to \RR^N$ as a function, and the derivative is well defined for each of its components as in the following. For each $\mu\in\mcP_2(\RR^N)$, there exists a unique version of such function $\partial_\mu U(\mu, \cdot)$ which is assumed to be a priori continuous (see the discussion in \cite{crisan2018smoothing}).

\paragraph{Multi-index.} To get a more general result, we extend the derivatives to higher order. For a function $f:\mcP_2(\RR^N) \to \RR^N$, we can apply the above discussion straightforwardly to each component $f=(f^1,\dots, f^N)$. Then the derivatives $\partial_\mu f^i, 1\le i\le N$ takes values in $\RR^N$, and  we denote $(\partial_\mu f^i)_j : \mcP_2(\RR^N)\times\RR^N \to \RR$ for $j=1,\dots, N$. For a fixed $v\in\RR^N$, we are able to differentiate $\mcP_2\ni \mu \mapsto (\partial_\mu f^i)_j(\mu, v) \in\RR$ again to get the second order derivative. If the derivative of this mapping exists and there is a continuous version of 
$$\mcP_2(\RR^N) \times \RR^N \times \RR^N \ni (\mu, v_1, v_2) \mapsto \partial_\mu(\partial_\mu f^i)_j(\mu, v_1, v_2) \in \RR^N,$$
then it is unique. It is natural to have a multi-index notation $\partial_\mu^{(j, k)}f^i:= (\partial_\mu(\partial_\mu f^i)_j)_k$ to ease the notation. Similarly, for higher derivatives, if for each $(i_0, \dots, i_n)\in\{1, \dots, N\}^{n+1}$,
$$\underbrace{\partial_\mu(\partial_\mu \dots (\partial_\mu}_{n \text{ times}} f^{i_0})_{i_1} \dots)_{i_n}$$
exists, we denote this $\partial_\mu^\alpha f^{i_0}$ with $\alpha=(i_1, \dots, i_n)$ and $|\alpha|=n$. Each derivative in $\mu$ is a function of an extra variable with $\partial_\mu^\alpha f^{i_0}: \mcP_2(\RR^N) \times (\RR^N)^n\to \RR$. We always denote these variables, by $v_1, \dots, v_n$, i.e.,  
$$ \mcP_2(\RR^N)\times (\RR^N)^n  \ni (\mu, v_1, \dots, v_n) \mapsto \partial_\mu^\alpha f^{i_0}(\mu, v_1, \dots, v_n) \in \RR.$$ 
When there is no confusion, we will abbreviate $(v_1, \dots, v_n)$ to $\bm{v}\in (\RR^N)^n$, so that 
$$\partial_\mu^\alpha f^{i_0}(\mu, \bm{v}) =\partial_\mu^\alpha f^{i_0}(\mu, v_1, \dots, v_n),$$
and use notation 
$$|\bm{v}| := |v_1| +\cdots + |v_n|,$$
with $|\cdot|$ the Euclidean norm on $\RR^N$. It then makes sense to discuss derivatives of the function $\partial^\alpha_\mu f^{i_0}$ with respect to variables $v_1, \dots, v_n$. 

If, for some $j\in \{1, \dots, N\}$ and all $(\mu, v_1, \dots, v_{j-1}, v_{j+1}, \dots, v_n) \in \mcP_2(\RR^N) \times (\RR^N)^{n-1}$, 
$$\RR^N \ni v_j \mapsto \partial_\mu^\alpha f^{i_0}(\mu, v_1, \dots, v_n) \in \RR$$
is $l$-times continuously differentiable, we denote the derivatives $\partial_{v_j}^{\beta_j}\partial_\mu^\alpha f^{i_0}$, for $\beta_j$ a multi-index on $\{1, \dots, N\}$ with $|\beta_j|\le l$. Similar to the above, we will denote by $\bm{\beta}$ the $n$-tuple of multi-indices $(\beta_1, \dots, \beta_n)$. We also associate a length to $\bm{\beta}$ by 
$$|\bm{\beta}| := |\beta_1| +\cdots + |\beta_n|,$$
and denote $\#\bm{\beta}:=n$. Then we denote by $\mc{B}_n$ the collection of all such $\bm{\beta}$ with $\#\bm{\beta} = n$, and $\mc{B} := \cup_{n\ge 1} \mc{B}_n$. Again, to lighten the notation, we use 
$$\partial_{\bm{v}}^{\bm{\beta}} \partial_\mu^\alpha f^i(\mu, \bm{v}) := \partial_{v_n}^{\beta_n}\cdots  \partial_{v_1}^{\beta_1} \partial_\mu^\alpha f^i(\mu, v_1,\dots, v_n).$$ 

The coefficients $V_0, \dots, V_d : [0,T]\times\RR^N \times\mcP_2(\RR^N) \times \RR^N \to \RR^N$ depend on a time variable, two Euclidean variables as well as the measure variable. So whether the order of taking derivatives matters is a question. Fortunately, a result from \cite[Lemma 4.1]{rainer2017} tells us that derivatives commute when the mixed derivatives are Lipschitz continuous. However, it should be emphasized that we could not interchange the order of $\partial_\mu$ and $\partial_v$, since the coefficients would not depend on the extra variable $\bm{v}$ before taking derivatives with respect to measure.

\begin{defn}[$\mcC^{k,k,k}_{b, \Lip}$] \label{def:Ckkk}
We have the following definitions:
\begin{enumerate}
\item[(a)] We use $\partial_x, \tilpartial$ to denote the derivative with respect to the second and fourth Euclidean variables in $V_0, V_i$'s, respectively.

\item[(b)] Let $V:\RR^+\times\RR^N\times \mcP_2(\RR^N) \times \RR^N \to \RR^N$ with components $V^1, \dots, V^N: \RR^+\times\RR^N\times \mcP_2(\RR^N) \times \RR^N \to \RR$. We say $V\in \mcC^{1,1,1}_{b, \Lip}([0,T]\times\RR^N\times \mcP_2(\RR^N) \times \RR^N ; \RR^N)$ if the following is true: for each $i=1,\dots, N$, $\partial_\mu V^i$, $\partial_x V^i$ and $\tilpartial V^i$ exist. Moreover, assume the boundedness of the derivatives for all $(t, x, \mu, y, v)\in [0,T]\times\RR^N \times \mcP_2(\RR^N) \times \RR^N \times \RR^N$,
$$|\partial_x V^i(t, x, \mu, y)| + |\tilpartial V^i(t, x, \mu, y)| + |\partial_\mu V^i(t, x, \mu, y, v)| \le C.$$
In addition, suppose that $\partial_\mu V^i$, $\partial_x V^i$ and $\tilpartial V^i$ are all Lipschitz in the sense that for all $(t, x, \mu, y, v), (t, x', \mu', y', v') \in [0,T]\times\RR^N\times \mcP_2(\RR^N) \times \RR^N $,
\begin{align*}
\big|\partial_\mu V^i(t, x, \mu, y, v) -\partial_\mu V^i(t, x', \mu', y', v') \big| &\le C(|x-x'| + |y-y'| + |v-v'| + W_2(\mu, \mu')), \\
\big|\partial_x V^i(t, x, \mu, y) -\partial_x V^i(t, x', \mu', y') \big| &\le C(|x-x'| + |y-y'| + W_2(\mu, \mu')), \\
\big|\tilpartial V^i(t, x, \mu, y) -\tilpartial V^i(t, x', \mu', y') \big| &\le C(|x-x'| + |y-y'| + W_2(\mu, \mu')),
\end{align*}
and $V^i$, $\partial_\mu V^i$, $\partial_x V^i$ and $\tilpartial V^i$ all have linear growth property,
\begin{align*}
&|V^i(t, x, \mu, y)| + |\partial_x V^i(t, x, \mu, y)| +|\partial_\mu V^i(t, x, \mu, y, v)|+ |\tilpartial V^i(t, x, \mu, y)| \\
&\hspace{-1.5em} \le C_T\big(1 + |x| + |y| + W_2(\mu, \mu_0) + |v|\big)
\end{align*}
for some fixed measure $\mu_0\in\mcP_2(\RR^N)$, and $C_T$ is a constant that depends only on $T$.

\item[(c)] We write $V\in \mcC^{k,k,k}_{b, \Lip}([0,T]\times\RR^N\times \mcP_2(\RR^N) \times \RR^N ; \RR^N)$, if the following holds true: for each $i=1, \dots, N$, and all multi-indices $\alpha$, $\tilde{\gamma}$ and $\gamma$ on $\{1, \dots, N\}$ and all $\bm{\beta}\in\mc{B}$ satisfying $|\alpha| +|\bm{\beta}|+|\gamma|+ |\tilde{\gamma}| \le k$, the derivative 
$$\partial_x^\gamma \tilpartial^{\tilde{\gamma}} \partial_{\bm{v}}^{\bm{\beta}} \partial_\mu^\alpha V^i(t, x, \mu, y, \bm{v})$$
exists and is bounded, Lipschitz continuous, and satisfies the linear growth condition. 

\item[(d)] We write $h \in \mcC^{k,k}_{b, \Lip}([0,T]\times\RR^N \times \RR^N ; \RR^N)$, if the mapping $h$ does not depend on a measure variable and all the other conditions are satisfied in (c).
\end{enumerate}
\end{defn}

\subsection{Solutions of Directed Chain SDEs}

The existence and uniqueness of weak solutions of directed chain SDEs are given in Proposition \ref{prop:sol of directed chain sde}. The constraint  \eqref{def:sde_random_coef_constrain} plays an essential role here to govern the uniqueness.

\begin{prop}
\label{prop:sol of directed chain sde}
Suppose that $V_i, i=0,1,\dots, d$ are Lipschitz in the sense that for every $T>0$, there exists a constant $C_T$ such that 
\begin{equation}
\label{def:lip}
\sup_{i} |V_i(t, x_1, \mu_1, y_1) - V_i(t, x_2, \mu_2, y_2)|\le C_T(|x_1-x_2| + |y_1-y_2| + W_2(\mu_1, \mu_2)), \quad 0\le t\le T.
\end{equation}
With the same constant $C_T$, let us also assume that $V_i$'s have at most linear growth, i.e.
\begin{equation}
	\sup_{0\le t\le T} |V_i(t, x, \mu, y)| \le C_T(1 + |x| + |y| + W_2(\mu, \mu_0))
\end{equation}
where $\mu_0\in\mcP_2(\RR^N)$ is fixed. Then there exists a unique weak solution to the directed chain stochastic differential equation \eqref{def:sde_random_coef}-\eqref{def:sde_random_coef_constrain}.
\end{prop}

The proof is similar to the proof for \cite[Proposition 2.1]{DETERING20202519} with a little generalization. Because of the appearance of the neighborhood process, we cannot expect a strong solution of the directed chain SDEs \eqref{def:sde_random_coef} (\textit{cf.} Proposition 2.1 of \cite{DETERING20202519}).

\begin{proof}
Let us first assume boundedness on all coefficients, i.e.
\begin{equation}
\label{def:lip+bdd}
\sup_{i} |V_i(t, x_1, \mu_1, y_1) - V_i(t, x_2, \mu_2, y_2)|\le C_T((|x_1-x_2| + |y_1-y_2| + W_2(\mu_1, \mu_2)) \wedge 1).
\end{equation}

We shall evaluate the Wasserstein distance between two probability measures $\mu_1, \mu_2$ on the space $C([0,T], \RR^N)$ of continuous functions, namely

\begin{equation}
\label{def:metric_process_measure}
D_t(\mu_1, \mu_2) :=\inf \bigg\{ \int(\sup_{0\le s\le t}|X_s(\omega_1) - X_s(\omega_2)|^2\wedge 1)\ud \mu(\omega_1, \omega_2) \bigg\}^{1/2}
\end{equation}
for $0\le t\le T$, where the infimum is taken over all the joint measure $\mu$ on $C([0,T], \RR^N)\times C([0,T], \RR^N)$ such that their marginals are $\mu_1, \mu_2$, and the initial joint distribution is $\theta \otimes \theta$, the initial marginals are $\theta$. Here, $X_s(\omega) = \omega(s), 0\le s\le T$ is the coordinate map of $\omega\in C([0,T], \RR^N)$. $D_T(\cdot, \cdot)$ defines a complete metric on $\mcM(C([0,T], \RR^N))$, which gives the weak topology to it. 

Given the distribution $m= \Law(\tilX) \in \mcM(C([0,T], \RR^N))$ of $\tilX$ that is independent of $B$ and $X_0$, it is well known that the following stochastic differential equation
\begin{equation}
\label{def:eq_xm}
	\ud X_t^m = V_0(t, X_t^m, m_t, \tilX_t) \ud t + \sum_{i=1}^d V_i(t, X_t^m, m_t, \tilX_t) \ud B_t^i
\end{equation} 
has a unique solution, based on the Lipschitz and linear growth condition on coefficients. Since $\tilX$ is independent of Brownian motion $B$, we can only expect the solution exists in weak sense. 

Define a map $\Phi: \mcM(C([0,T], \RR^N))\to \mcM(C([0,T], \RR^N))$ by $\Phi(m) := \Law(X_\cdot^m)$. We shall find a fixed point $m^*$ for the map $\Phi$ such that $\Phi(m^*)=m^*$ to show the uniqueness of the solution in the weak sense.

Assume $m_1=\Law(\tilX^1)$ and $ m_2=\Law(\tilX^2)$, then by rewriting \eqref{def:eq_xm} we have
$$ X^{m_i}_t = \theta + \int_0^t V_0(t, X_t^{m_i}, m_{i, t}, \tilX^i_t) \ud s + \sum_{i=1}^d\int_0^t V_i(t, X_t^{m_i}, m_{i, t}, \tilX^i_t) \ud B^i_s, \quad i=1,2.$$

Note that here we fix the initial state to be the same $\theta$ for both $X^{m_1}$ and $X^{m_2}$. Let $m$ be a joint distribution of $m_1, m_2$ and $\EE^m$ be the expectation under $m$.
Under the stronger assumption \eqref{def:lip+bdd},
\begin{align}
 \EE^m\big[\sup_{0\le s\le t} (X_s^{m_1} - X_s^{m_2})^2\big] &\le 2 \EE^m\bigg[\sup_{0\le s\le t} \int_0^s \big(V_0(v, X_v^{m_1}, m_{1, v}, \tilX^1_v) - V_0(v, X_v^{m_2}, m_{2, v}, \tilX^2_v)\big)^2 \ud v\bigg] \nonumber \\
& \hspace{-2em} + 2^d \sum_{i=1}^d\EE^m\bigg[\sup_{0\le s\le t} \int_0^s \big(V_i(v, X_v^{m_1}, m_{1, v}, \tilX^1_v) - V_i(v, X_v^{m_2}, m_{2, v}, \tilX^2_v)\big)^2 \ud v\bigg] \nonumber \\
&\hspace{-8em} \le 2^{d+3}(d+1) C_T \EE^m\bigg[ \sup_{0\le s\le t} \int_0^s \big((X_v^{m_1}- X_v^{m_2})^2 + W_2(m_{1, v}, m_{2, v})^2 + (\tilX^1_v - \tilX_v^2)^2\big)\wedge 1 \ud v \bigg] \nonumber \\
&\hspace{-8em}\le C\cdot \EE^m\bigg[\int_0^t \sup_{0\le v\le s} (X_v^{m_1}- X_v^{m_2})^2\wedge 1 \ud s \bigg]  + C \int_0^t W_2(m_{1, s}, m_{2, s})^2\wedge 1 \ud s \nonumber \\
& + C\cdot \EE^m\bigg[\int_0^t \sup_{0\le v\le s} (\tilX^1_v - \tilX_v^2)^2\wedge 1 \ud s \bigg]  \nonumber \\
& \hspace{-8em} = C\int_0^t \EE^m \big[ \sup_{0\le v\le s} (X_v^{m_1}- X_v^{m_2})^2\wedge 1\big] \ud s   + C \int_0^t W_2(m_{1, s}, m_{2, s})^2\wedge 1 \ud s \nonumber \\
& + C\int_0^t \EE^m\big[ \sup_{0\le v\le s} (\tilX^1_v - \tilX_v^2)^2\wedge 1\big] \ud s \label{def:temp_eq}
\end{align}
where we replace $2^{d+3}(d+1) C_T$ by $C$.
Note that by construction,
$$W_2(m_{1, s}, m_{2, s})^2\wedge 1 \le D_s(m_1, m_2)^2.$$

By taking infimum over all $m$ such that its marginals are $m_1, m_2$, the third term in \eqref{def:temp_eq} is bounded by 
$$C\int_0^t D_s(m_1, m_2)^2 \ud s.$$

Hence we get
$$D_t(\Phi(m_1), \Phi(m_2))^2 \le C \int_0^t D_s(\Phi(m_1), \Phi(m_2))^2 \ud s + 2C\int_0^t D_s(m_1, m_2)^2\ud s.$$

Then by applying Gronwall's lemma, we get
\begin{equation}
\label{def:eq_contraction}
D_t(\Phi(m_1), \Phi(m_2))^2 \le 2C e^{CT}\int_0^t D_s(m_1, m_2)^2\ud s. 
\end{equation}

For every $m\in  \mcM(C([0,T], \RR^N))$, let $m_1 = m$, $m_2 = \Phi(m)$, we get by iterating \eqref{def:eq_contraction}, 
\begin{equation}
D_T(\Phi^{(k+1)}(m), \Phi^{(k)}(m)) \le \sqrt{\frac{(2CT e^{CT})^k}{k!}} D_T(\Phi(m), m), \quad \forall k \in \N.
\end{equation}

This implies that $\{\Phi^{(k)}(m), k\in \N\}$ forms a Cauchy sequence converging to a fixed point $m^*$. This $m^*$ is the weak solution to directed chain SDE \eqref{def:sde_random_coef}\&\eqref{def:sde_random_coef_constrain}. To relax the bounded condition \eqref{def:lip+bdd} to \eqref{def:lip}, we can first cut $[0,T]$ to small time intervals such that the bounded assumption is satisfied on each interval, and establish the uniqueness on each interval and finally paste them together.
\end{proof}

\begin{prop}[Regularity]
\label{prop:regularity_Xtheta}
If $\theta\in L^2(\Omega)$, the solution of directed chain SDE \eqref{def:sde_random_coef}-\eqref{def:sde_random_coef_constrain} satisfies 
$$\| X^{\theta}\|_{\mcS^2_T} \le C(1 + \|\theta\|_2),$$
where $C = C(T)$, under the assumption of Proposition \ref{prop:sol of directed chain sde}. 
\end{prop}
\begin{proof}
The proof follows from a similar procedures as \cite[Proposition 2.2]{DETERING20202519}.
\end{proof}

\subsection{Flow Property}
\label{sec:flow}
In the last part of this section, we discuss the flow property of directed SDEs informally. After establishing the solution to the exact directed chain SDE \eqref{def:sde_random_coef}, we also consider the process $X^{x, [\theta]}_\cdot$ that satisfies 
\begin{equation}
\label{def:sde_random_coef_x}
	X_\cdot^{x, [\theta]} = x + \int_0^\cdot V_0(s, X_s^{x, [\theta]}, [X_s^\theta], \tilX_{s}) \ud s + \sum_{i=1}^d \int_0^\cdot V_i(s, X_s^{x, [\theta]}, [X_s^\theta], \tilX_{s}) \ud B_s^i,
\end{equation}
where $x\in\RR^N$ is a fixed initial point and $\tilX_\cdot$ is the neighborhood process satisfying the constraints \eqref{def:sde_random_coef_constrain}, i.e., $\Law(\tilX_\cdot) = \Law(X^{\theta}_\cdot)=[X^{\theta}_\cdot]$. Note that $X_\cdot^{x, [\theta]}$ in \eqref{def:sde_random_coef_x} is strongly solvable with pathwise uniqueness, given the unique, weak solution $(X^\theta_\cdot, \tilX_\cdot, B_\cdot)$ as in Proposition \ref{prop:sol of directed chain sde}. 

\begin{prop}[Regularity]
Under the assumption in Proposition \ref{prop:sol of directed chain sde}, for every $\theta\in L^2(\Omega)$, $T > 0$ and $p \ge 2$, there exists  a constant $C = C(T, p)$ such that the solution of 
\eqref{def:sde_random_coef_x} satisfies 
$$\| X^{x, [\theta]}\|_{\mcS^p_T} \le C(1 + \|\theta\|_2 + |x|).$$
\end{prop}
\begin{proof}
The proof follows from the Burkholder-Davis-Gundy inequality and Proposition \ref{prop:regularity_Xtheta}, which is also satisfied by $\tilX$.
\end{proof}

For the explanation purpose, we will add a superscript $\tilde{\theta}$ such that $X_t^{x, \theta, \tilde{\theta}}:= X_t^{x, \theta}$ and $\tilX^{\tilde{\theta}}_t:=\tilX_t$  to emphasize the neighborhood process start at $\tilde{\theta}$, independent of $\theta$. This notation is only used in this subsection. Thus, with the notation $B^0_t\equiv t$, $t \ge 0$, \eqref{def:sde_random_coef_x} is read as

\begin{equation} \label{eq:sde_random_coef_x2} 
	X_t^{x, [\theta], \tilde{\theta}} = x + \sum_{i=0}^d \int_0^t V_i(s, X_s^{x, [\theta], \tilde{\theta}}, [X_s^\theta], \tilX_{s}^{\tilde{\theta}}) \ud B_s^i, \quad t \ge 0 . 
\end{equation}

For different initial points $x x^\prime$ and the corresponding solutions $X_\cdot^{x, [\theta], \tilde{\theta}}$ and $X_\cdot^{x^\prime, [\theta], \tilde{\theta}}$,  
we have the following estimate: there exists a constant $ C>0$ such that 
$$\EE\big[ \sup_{t\le s\le T} \big|X_s^{x, [\theta], \tilde{\theta}} - X_s^{x', [\theta], \tilde{\theta}}\big|^2 \big] \le C |x-x'|^2$$
again by the Lipschitz continuity and the Burkholder-Davis-Gundy inequality. By the pathwise uniqueness of $X_\cdot^{x, [\theta], \theta}$, given the pair $(X_\cdot^{\theta}, \tilde{X}^{\tilde{\theta}}_\cdot)$, it follows
\begin{equation} \label{eq:sde_random_coef_x3}
    X_s^{x, [\theta], \tilde{\theta}}\bigg|_{x=\theta} = X_s^\theta, \quad 0\le s\le T.
\end{equation}

Now, with some abuse of notations, we denote by $X^{t, x, [\theta], \tilde{\theta}}_\cdot$ the solution to \eqref{eq:sde_random_coef_x2} with $X_t^{t, x, [\theta], \tilde{\theta}} = x$, denote by $(X^{t, \theta}_\cdot, \tilX^{t, \tilde{\theta}}_\cdot )$ the solution to \eqref{def:sde_random_coef} with $(X^{t, \theta}_t, \tilX^{t, \tilde{\theta}}_t ) = ( \theta, \tilde{\theta} )$. It follows from \eqref{eq:sde_random_coef_x3} that by the strong Markov property, for $0 \le t \le s \le r \le T$, we have the flow property 
\begin{equation}
\label{def:flow}
    (X_r^{s, X_s^{t, x, [\theta], \tilde{\theta}}, [X_s^{t, \theta}], \tilX_s^{t, \tilde{\theta}}}, X_{r}^{s, X_s^{t, \theta}}, \tilX_r^{s, \tilX_s^{t, \tilde{\theta}}}) = (X_r^{t, x, [\theta], \tilde{\theta}}, X_r^{t, \theta}, \tilX_r^{t, \tilde{\theta}}).
\end{equation}

We close section \ref{sec:directed_sde} at this point. After the introduction of the partial Malliavin derivatives, we will revisit the directed chain SDE and study the regularities of its derivatives.

%% file: malliavin.tex
In this section, we will briefly review the Malliavin calculus, following \cite{nualart2009malliavin}, and introduce the partial Malliavin derivatives for our problem.

\paragraph{Malliavin Calculus.} Let $H := L^2([0,T], \RR^d)$ be the Hilbert space, where we define Gaussian process, and $\mcS$ be the set of smooth functionals of the form
$$F(\omega) = f\bigg(\int_0^Th_1(t)\cdot \ud B_t(\omega), \dots, \int_0^Th_n(t)\cdot \ud B_t(\omega)\bigg),$$
where $f\in \mcC_p^\infty(\RR^n; \RR)$ and $\int_0^Th_i(t)\cdot \ud B_t = \sum_{j=1}^d \int_0^T h_i^j(t) \ud B^j_t$.

Then the Malliavin derivative of $F$, denoted by ${\bm D}F \in L^2(\Omega; H)$ is given by:
\begin{equation}
{\bm D}F = \sum_{i=1}^n \partial^i f \bigg(\int_0^Th_1(t)\cdot \ud B_t(\omega), \dots, \int_0^Th_n(t)\cdot \ud B_t(\omega)\bigg) h_i.
\end{equation}

As stated in \cite{nualart2009malliavin}, because of the isometry $L^2(\Omega\times[0,T]; \RR^d) \simeq L^2(\Omega; H)$, we are able to identify ${\bm D}F$ with a process $({\bm D}_rF)_{r\in[0,T]}$ taking values in $\RR^d$. Moreover, the set of smooth functionals, denoted by $\mcS$, is dense in $L^p(\Omega)$ for any $p\ge 1$ and ${\bm D}$ is closable as operator from $L^p(\Omega)$ to $L^p(\Omega; H)$. We define $\bD^{1, p}$ as the closure of the set $\mcS$ within $L^p(\Omega; \RR^d)$ with respect to the norm
$$\|F\|_{\bD^{1, p}} = \big( \EE |F|^p + \EE \| \bm{D} F \|_{H}^p \big)^{\frac{1}{p}}.$$

The higher order Malliavin derivatives are defined similarly, denoted by $\bm{D}^{(k)} F$, which is a random variable with values in $H^{\otimes k}$ defined as
$$\bm{D}^{(k)} F := \sum_{i_1, \dots, i_k=1}^n \partial^{(i_1, \dots, i_k)} f \bigg(\int_0^Th_1(t)\cdot \ud B_t(\omega), \dots, \int_0^Th_n(t)\cdot \ud B_t(\omega)\bigg).$$
We define $\bD^{k, p}$ to be the closure of the set of smooth functions $\mcS$ with respect to the norm:
$$\|F\|_{\bD^{k, p}} = \big( \EE |F|^p + \sum_{j=1}^k\EE \| \bm{D}^{(j)} F \|_{H}^p \big)^{\frac{1}{p}}.$$

The Malliavin derivative is also well defined for the general $E$-valued random variables, where $E$ is some separable Hilbert space, and we write $\bD^{1, p}(E)$ to be the closure of $\mcS$ under some appropriate metric with respect to $E$. We will use notation $\bD^{1, \infty}=\cap_{p\ge 1} \bD^{1, p}$.
 The adjoint operator of $\bm{D}$ is introduced as follows.
 \begin{defn}[Definition 1.3.1, \cite{nualart2009malliavin}]
 We denote by $\delta$ the adjoint of the operator $\bm{D}$. That is, $\delta$ is an unbounded operator on $L^2(\Omega; H)$ with values in $L^2(\Omega)$ such that
 \begin{enumerate}
 \item The domain of $\delta$, denoted by $\Dom \, \delta$, is the set of $H$-valued square integrable random variables $u\in L^2(\Omega; H)$ such that
 $$\big| \EE[\langle \bm{D}F, u \rangle_H ]\big| \le c \|F\|_2$$
 for all $F\in \bD^{1,2}$, where $c$ is some constant depending on $u$.
 \item If $u$ belongs to $\Dom \, \delta$, then $\delta(u)$ is the element of $L^2(\Omega)$ characterized by $$\EE[F\delta(u)] = \EE[\langle \bm{D}F, u\rangle_H]$$
 for any $F\in\bD^{1,2}$.
 \end{enumerate}
 \end{defn}
 
\subsection{Partial Malliavin Calculus}
The following remark motivate us to use partial Malliavin calculus.
\begin{remark}
Because of the appearance of a neighborhood process $\widetilde{X}_\cdot$, we propose the following problem. We shall remark that almost everything satisfied by the McKean-Vlasov SDE in \cite{crisan2018smoothing} is also satisfied by our directed chain SDE. However, we cannot directly apply their approach to argue the existence, contituity and differentiability of the density function of $X^{x, [\theta]}_t$. The reason is that a key step connecting the Malliavin derivative and $\partial_x X^{x, [\theta]}_t$, which is defined in \eqref{def:sde_random_coef_x}, may not hold in our case, i.e., in general, the identity
\begin{equation}
\label{eq: question}
\partial_x X_t^{x, [\theta]} = \bm{D}_r X_t^{x, [\theta]} \sigma^\top\big(\sigma \sigma^\top\big)^{-1}(r, X_r^{x, [\theta]}, [X^\theta_r], \tilX_r) \partial_x X_r^{x, [\theta]}
\end{equation}
does not hold for any $r\le t$. Thus, we cannot directly make use of the  integration by parts formulae in \cite{crisan2018smoothing}, and hence, we cannot argue the smoothness of $X^{x,[\theta]}_t$.
\end{remark}

\begin{question}
\label{Q:integral_by_parts}
How can we make connections between the first order derivative $\partial_x X^{x, [\theta]}_t$ and the Malliavin derivatives similar to \eqref{eq: question}, which would render us to apply integration by parts formula?
\end{question}

To address Question \ref{Q:integral_by_parts}, we consider the  partial Malliavin derivative  in \cite{nualart1989partial}. Let $\mc{G} := \sigma(\{\tilX_{t_i}, \forall t_i\in \bQ_T\})$ be the sigma algebra generated by the neighborhood process at all rational time, where $\bQ_T = \bQ\cap[0,T]$ denote the collection of all rational numbers in $[0,T]$. Due to the continuity of $\tilX$,  considering all rational time stamps is equivalent to considering the whole time interval $[0, T]$, i.e. $\mc{G} = \sigma(\tilX_s, 0\le s\le T)$. We associate to $\mc{G}$ the family of subspaces defined by the orthogonal complement to the subspace generated by $\{\bm{D} \tilX_{t_i}(\omega), t_i\in \bQ_T\}$, i.e., 
$$K(\omega) = \langle \bm{D}\tilX_{t_i}(\omega), t_i\in\bQ_T \rangle^\perp.$$

Since $\mc{G}$ is generated by countably many random variables, we say it is countably smooth generated. Then the family $\mc{H} := \{K(\omega), \omega\in \Omega\}$ has a measurable projection by this countably smoothness of $\mc{G}$. We can define the partial Malliavin derivative operator as $\bm{D}^{\mc{H}}$.

\begin{defn}[Definition 2.1, \cite{nualart1989partial}]
\label{def:D_H}
We define the partial derivative operator $\bm{D}^{\mc{H}} : \bD^{1,2} \to L^2(\Omega, H)$ as the projection of $\bm{D}$ on $\mc{H}$, namely, for any $F\in \bD^{1, 2}$,
$$\bm{D}^{\mc{H}} F = \Proj_{\mc{H}}(\bm{D}F) =  \Proj_{K(\omega)}(\bm{D}F)(\omega).$$

This operator, similar to $\bm{D}$, admits an identification with a process $(\bm{D}^{\mc{H}}_r)_{r\in [0,T]}$. Moreover, we define the norm associated with $\bm{D}^{\mc{H}} $ by
$$\|F\|_{\bD_{\mc{H}}^{k, p}} = \big( \EE |F|^p + \sum_{j=1}^k\EE \| \bm{D}^{{\mc{H}}, (j)} F \|_{H}^p \big)^{\frac{1}{p}},$$
where $\bm{D}^{{\mc{H}}, (j)}$ is defined as
$$\bm{D}^{{\mc{H}}, (j)} F = \Proj_{\mc{H}}(\bm{D}^{(j)}F) =  \Proj_{K(\omega)}(\bm{D}^{(j)} F)(\omega).$$
\end{defn}

Now we have the important fact that $\bm{D}^{\mc{H}} \tilX_{t} = 0$. This is because $\tilX_{t}$ is $\mc{G}$ measurable and hence equivalently 
\begin{equation}
\bm{D} \tilX_t \in \langle \bm{D}\tilX_{t_i}, t_i\in\bQ_T \cup \{t\}   \rangle; \quad t \in [0, T] . 
\end{equation}

Then the projection of $\bm{D} \tilX_t$ on to the orthogonal of $ \langle \bm{D}\tilX_{t_i}, t_i\in\bQ_T \cup \{t\}   \rangle$ must be zero. 
Similar to the common Malliavin calculus, we have an adjoint operator of $\bm{D}^{\mc{H}}$, which is denoted by $\delta_{\mc{H}}$, as well as the integration by parts formula for the partial Malliavin calculus.

\begin{defn}[Definition 2.3, \cite{nualart1989partial}]
\label{def:delta_H}
Set $\Dom\, \delta_{\mc{H}} = \{ u \in L^2(\Omega; H): \Proj_{\mc{H}} u \in \Dom\, \delta\}$. For any $u\in\Dom\, \delta_{\mc{H}}$, set $\delta_{\mc{H}}( u )= \delta( \Proj_{\mc{H}} u)$.
\end{defn}

Following Definition \ref{def:D_H} and \ref{def:delta_H}, we have integration by parts formula for $\bm{D}^{\mc{H}}$ and $\delta_{\mc{H}}$
\begin{equation}
\label{def:eq_partialMalliavin_IBPF}
    \EE[\langle h, \bm{D}^{\mc{H}} F \rangle] = \EE[\langle \Proj_{\mc{H}} h, \bm{D}F \rangle] = \EE[F \delta_{\mc{H}}(h)].
\end{equation}



\paragraph{Kusuoka-Stroock Processes}
In order to derive the differentiability of the density function, we mimic the procedure in \cite{crisan2018smoothing} and need to develop the integration-by-parts formulae introduced in the works of \cite{kusuoka20032} and \cite{kusuoka1987applications}.

\begin{defn}[Definition 2.8 in \cite{crisan2018smoothing}]
Let $E$ be a separable Hilbert space and let $r\in\RR$, $q, M\in\N$. We denote by $\bK_r^q(E, M)$ the set of processes $\Psi: [0,T]\times \RR^N\times \mcP_2(\RR^N)  \to \bD^{M, \infty}(E)$ satisfying the following:

\begin{enumerate}
\item For any multi-indices $\alpha, {\bm \beta}, \gamma$ satisfying $|\alpha|+|{\bm \beta}| + |\gamma| \le M$, the function
$$ [0,T]\times\RR^N\times\mcP_2(\RR^N)  \ni (t, x, [\theta]) \mapsto \partial_x^\gamma \partial^{\bm \beta}_{\bm v} \partial_\mu^\alpha \Psi(t, x, [\theta], v) \in L^p(\Omega)$$
exists and is continuous for all $p\ge 1$.
\item For any $p\ge 1$ and $m\in\N$ with $|\alpha|+|{\bm \beta}| + |\gamma|+m \le M$, we have
$$ \sup_{v\in (\RR^N)^{\#\bm \beta}} \sup_{t\in (0,T]} t^{-r/2} \bigg\| \partial_x^\gamma \partial^{\bm \beta}_{\bm v} \partial_\mu^\alpha \Psi(t, x, [\theta], v) \bigg\|_{\bD_{\mc{H}}^{m, p}(E)} \le C\,(1+|x|+\|\theta\|_2)^q$$
\end{enumerate}
\end{defn}

In our discussion, we do not consider the differentiability of the process $X$ with respect to the initial state of its neighborhood $\tilX$. This above definition of $\bK_r^q(E)$ is almost the same as the definition in \cite[Definition 2.8]{crisan2018smoothing}, except for the norm. The reason is that we only care about the existence and smoothing properties of the density function of $X^{x, [\theta]}$ and have to use the  partial Malliavin calculus. We remark that although the norms are different, all the regularity results under the norm $\|\cdot\|_{\bD^{k, p}}$ also holds under our norm $\|\cdot\|_{\bD_{\mc{H}}^{k, p}}$ because of the H\"{o}lder's inequality. To get the smoothness of density functions of a process start from a fixed initial point, we use $\mc{K}_r^q(\RR, M)$ as the class of Kusuoka-Stroock processes which do not depend on a measure term. By \cite[Lemma 2.11]{crisan2018smoothing}, if $\Psi \in \bK_r^q(E, M)$, then $\Phi(t, x, y) := \Psi(t, x, \delta_x, y) \in \mc{K}_r^q(E, M)$. 

%% file: density.tex
\subsection{Regularities of Solutions of Directed Chain SDEs}
For the purpose of establishing the integration by parts formulae for the directed chain SDEs and applying the results in \cite[Theorem 6.1]{crisan2018smoothing}, we only need to check all the regularities conditions with respect to parameters ($\theta, x$) contained in \cite[Section 3]{crisan2018smoothing}.

\begin{prop}[First-order derivatives]
\label{prop:first-order derivatives}
Suppose that $V_0, \dots, V_d \in \mcC^{1,1,1}_{b, Lip}(\RR^+\times\RR^N \times \mcP_2(\RR^N)\times \RR^N; \RR^N)$. Then the following statements hold:
\begin{enumerate}
	\item There exists a modification of $X^{x, [\theta]}$ such that for all $t\in[0,T]$, the map $x\mapsto X_t^{x, [\theta]}$ is $\PP$-a.s. differentiable. We denote the derivative by $\partial_x X^{x, [\theta]}$ and note that it solves the following SDE
	\begin{align}
	\label{def:first_order_x}
	\partial_x X^{x, [\theta]}_t &= Id_N + \sum_{i=0}^d\int_0^t \bigg\{\partial V_i(s, X_s^{x, [\theta]}, [X_s^\theta], \tilX_s) \partial_x X_s^{x, [\theta]} \bigg\} \ud B_s^i 
	\end{align}
	for every $t \in [0, T] $. 
	\item For all $t\in[0,T]$, the maps $\theta\mapsto X_t^\theta$ and $\theta\mapsto X_t^{x, [\theta]}$ are Fr\'echet differentiable in $L^2(\Omega)$, i.e. there exists a linear continuous map $\mc{D} X_t^\theta: L^2(\Omega) \to L^2(\Omega)$ such that for all $\gamma\in L^2(\Omega)$, 
	$$\| X_t^{\theta+\gamma} - X_t^{\theta} - \mc{D} X_t^\theta(\gamma) \|_2 = o(\|\gamma\|_2) \quad \text{as} \quad \, \|\gamma\|_2\to 0,$$
	and similarly for $X_t^{x, [\theta]}$. These processes satisfy the following stochastic differential equations
	\begin{align}
	\mc{D} X_t^{x, [\theta]}(\gamma) &= 
	 \sum_{i=0}^d \int_0^t\bigg[ \partial V_i(s, X_s^{x, [\theta]}, [X_s^\theta], \tilX_s)\mc{D}X_s^{x, [\theta]}(\gamma) +  \tilpartial V_i(s, X_s^{x, [\theta]}, [X_s^\theta], \tilX_s)\mc{D}\tilX_s(\gamma)\nonumber\\
	&\hspace{4em}+ \mc{D} V_i'(s, X_s^{x, [\theta]}, X_s^\theta, \tilX_s)(\mc{D}X_s^\theta(\gamma))  \bigg] \ud B_s^i, \label{def:frechet_derivative1}
	\end{align}
	\begin{align}
	\mc{D} X_t^{\theta}(\gamma) &=  \gamma +
	 \sum_{i=0}^d \int_0^t\bigg[ \partial V_i(s, X_s^{\theta}, [X_s^\theta], \tilX_s)\mc{D}X_s^{\theta}(\gamma) +  \tilpartial V_i(s, X_s^{\theta}, [X_s^\theta], \tilX_s)\mc{D}\tilX_s(\gamma)\nonumber\\
	&\hspace{4em}+ \mc{D} V_i'(s, X_s^{\theta}, X_s^\theta, \tilX_s)(\mc{D}X_s^\theta(\gamma))  \bigg] \ud B_s^i \label{def:frechet_derivative2}
	\end{align}
	where $\widetilde{V_i}$ is the lifting of $V_i$. Moreover, for each $x\in\RR^N$, $t\in [0,T]$, the map $\mcP_2 \ni [\theta]\mapsto X_t^{x, [\theta]} \in L^p(\Omega)$ is differentiable for all $p\ge 1$. So, $\partial_\mu X_t^{x, [\theta]}(v)$ exists and it satisfies the following equation
	\begin{align}
	\partial_\mu X_t^{x, [\theta]}(v) &= \sum_{i=0}^d \int_0^t \bigg\{ \partial V_i \big(s, X_s^{x, [\theta]}, [X_s^\theta], \tilX_s \big) \partial_\mu X_s^{x, [\theta]}(v) \nonumber \\
	& \hspace{2em} + \tilpartial V_i \big(s, X_s^{x, [\theta]}, [X_s^\theta], \tilX_s \big) \partial_\mu \tilX_s(v) \nonumber \\
	& \hspace{2em} + \EE'\bigg[\partial_\mu V_i\big( s, X_s^{x, [\theta]}, [X_s^\theta], \tilX_s, (X_s^{v, [\theta]})' \big) \partial_x (X_s^{v, [\theta]})'\bigg] \nonumber \\
	& \hspace{2em} + \EE'\bigg[\partial_\mu V_i\big( s, X_s^{x, [\theta]}, [X_s^\theta], \tilX_s, (X_s^{\theta'})' \big) \partial_\mu (X_s^{\theta', [\theta]})'(v)\bigg] \bigg\} \ud B_s^i,	 \label{def:eq_partialmuXx}
	\end{align} 
	where $(X_s^{\theta'})'$ is a copy of $X_s^\theta$ on the probability space $(\Omega', \mcF',\PP')$. Similarly, $ \partial_x(X_s^{v, [\theta]})'$ is a copy of $\partial_x X_s^{v, [\theta]}$ and $\partial_\mu (X_s^{\theta', [\theta]})'=\partial_\mu (X_s^{x, [\theta]})'\big|_{x=\theta'}$. Finally, the following representation holds for all $\gamma \in L^2(\Omega)$:
	\begin{equation}
	\label{def:eq_DXx[theta]}
		\mc{D}X_t^{x, [\theta]}(\gamma) = \EE'[\partial_\mu X_t^{x, [\theta]}(\theta')   \gamma'].
	\end{equation}
	
	\item For all $t\in [0,T]$, $X_t^{x, [\theta]}, X_t^\theta\in \bD^{1, \infty}$. Moreover, $\bm{D}^{\mc{H}}_r X^{x, [\theta]} = \bigg(\bm{D}^{\mc{H}, j}_r(X^{x, [\theta]})^i\bigg)_{\substack{1\le j\le N\\ 1\le i\le d}}$ satisfies, for $0\le r\le t$
	\begin{align}
		\bm{D}^{\mc{H}}_r X_t^{x, [\theta]} &= \sigma\big(r, X_r^{x, [\theta]}, [X_r^\theta], \tilX_r \big)+ \sum_{i=0}^d \int_r^t \bigg(\partial V_i(s, X_s^{x, [\theta]}, [X_s^\theta], \tilX_s)\bm{D}^{\mc{H}}_r X_s^{x, [\theta]} \bigg) \ud B_s^i,
		\label{def:sde_DX}
	\end{align}
where $ \sigma\big(r, X_r^{x, [\theta]}, [X_r^\theta], \tilX_r \big)$ is the $N\times d$ matrix with columns $V_1,\dots, V_d$.
\end{enumerate}
\end{prop}

\begin{proof}
\begin{enumerate}
	\item The SDE of $X^{x, [\theta]}$ satisfies a classical SDE with adapted coefficients, by \cite[Theorem 7.6.5]{kunita1984stochastic} there exists a modification of $X_t^{x, [\theta]}$ which is continuously differentiable in $x$, and the first derivative satisfies \eqref{def:first_order_x}.
	\item The maps $\theta \mapsto X_t^\theta$ and $\theta \mapsto X_t^{x, [\theta]}$ are Fr\'echet differentiable by \cite[Lemma 4.17]{chassagneux2014probabilistic}. Then \eqref{def:frechet_derivative1} and \eqref{def:frechet_derivative2} follow from direct computation.
	
	Let us first rewrite the equation for $\mc{D}X_t^\theta(\gamma)$ in terms of the lifting $V'$,
	
	\begin{align}
	\mc{D}X_t^\theta(\gamma) &= \gamma + \sum_{i=0}^d \int_0^t\bigg[ \partial V_i(s, X_s^{\theta}, [X_s^\theta], \tilX_s)\mc{D}X_s^{\theta}(\gamma) +  \tilpartial V_i(s, X_s^{\theta}, [X_s^\theta], \tilX_s)\mc{D}\tilX_s(\gamma)\nonumber\\
	&\hspace{4em}+ \EE'\big[\partial_\mu V_i'(s, X_s^{\theta}, [X_s^\theta], \tilX_s, (X_s^{\theta'})')(\mc{D}(X_s^{\theta'})'(\gamma')) \big] \bigg] \ud B_s^i.  \label{def:eq_DXtheta}
	\end{align}
	
	We then consider the equation that we are going to prove for $\partial_\mu X_s^{\theta', [\theta]}(v)$, evaluated at $v=\theta''$ and multiplied by $\gamma''$ with both random variables defined on a probability space $(\Omega'', \mcF'', \PP'')$. Then taking expectation with respect to $\PP''$, we get
	\begin{align}
	\EE''\big[ \partial_\mu X_t^{\theta', [\theta]}(\theta'')\gamma'' \big] &=  \sum_{i=0}^d \int_0^t\bigg\{\partial V_i(s, X_s^{\theta}, [X_s^\theta], \tilX_s) \EE''[\partial_\mu X_s^{\theta', [\theta]}(\theta'')\gamma''] \nonumber\\
	&\hspace{1em} + \tilpartial V_i(s, X_s^{\theta}, [X_s^\theta], \tilX_s)\EE''[\partial_\mu \tilX_s \gamma''] \nonumber \\
	&\hspace{1em} + \EE'' \EE' \bigg[ \partial_\mu V_i(s, X_s^{\theta}, [X_s^\theta], \tilX_s, (X_s^{\theta'', [\theta]})') \partial_x (X_s^{\theta'', [\theta]})' \gamma'' \bigg] \nonumber \\
	&\hspace{1em} +\EE' \big[ \partial_\mu V_i(s, X_s^{\theta}, [X_s^\theta], \tilX_s, (X_s^{\theta'})') \EE''[ \partial_x (X_s^{\theta', [\theta]})'(\theta'') \gamma'' ]\big] \bigg\} \ud B^i_s.
	\end{align}
	
	Note that since $(\gamma'', \theta'')$ are defined on a separate probability space, we have $\EE''[\partial_\mu \tilX_s \gamma''] = \mc{D} \tilX_s(\gamma)$ and
	\begin{align*}
	\EE'' \EE' \big[ \partial_\mu V_i(s, X_s^{\theta}, [X_s^\theta], \tilX_s, (X_s^{\theta'', [\theta]})') \partial_x (X_s^{\theta'', [\theta]})' \gamma'' \big] &= \\
	& \hspace{-8em}\EE'[ \partial_\mu V_i(s, X_s^{\theta}, [X_s^\theta], \tilX_s, (X_s^{\theta'})') \partial_x (X_s^{\theta', [\theta]})' \gamma'  ].
	\end{align*}
	
	Then the dynamic of $\EE''[ \partial_\mu X_t^{\theta', [\theta]}(\theta'')\gamma'']$ reduces to
	\begin{align}
	\EE''\big[ \partial_\mu X_t^{\theta', [\theta]}(\theta'')\gamma'' \big] &=  \sum_{i=0}^d \int_0^t\bigg\{\partial V_i(s, X_s^{\theta}, [X_s^\theta], \tilX_s) \EE''[\partial_\mu X_s^{\theta', [\theta]}(\theta'')\gamma''] \nonumber\\
	&\hspace{2em} + \tilpartial V_i(s, X_s^{\theta}, [X_s^\theta], \tilX_s) \mc{D} \tilX_s(\gamma)  \nonumber \\
	&\hspace{-6em} +\EE' \big[ \partial_\mu V_i(s, X_s^{\theta}, [X_s^\theta], \tilX_s, (X_s^{\theta'})') \big[\partial_x (X_s^{\theta', [\theta]})' \gamma'  +  \EE''[ \partial_x (X_s^{\theta'', [\theta]})'(\theta'') \gamma'' ] \big]\bigg\} \ud B^i_s.
	\end{align}
	
	By \eqref{def:first_order_x}, we can evaluate the equation at $x=\theta$, multiply by x, and derive a dynamic of $\partial_x X_t^{\theta, [\theta]}\gamma$. It can be seen that $\partial_x X_t^{\theta, [\theta]}\gamma + \EE''[\partial_\mu X_t^{\theta', [\theta]}(\theta'')\gamma'']$ is equal to
	\begin{align}
	&\gamma + \sum_{i=0}^d \int_0^t\bigg\{\partial V_i(s, X_s^{\theta}, [X_s^\theta], \tilX_s) \EE''[\partial_\mu X_s^{\theta', [\theta]}(\theta'')\gamma'']  + \tilpartial V_i(s, X_s^{\theta}, [X_s^\theta], \tilX_s) \mc{D} \tilX_s(\gamma)  \nonumber \\
	&\hspace{0em} +\EE' \big[ \partial_\mu V_i(s, X_s^{\theta}, [X_s^\theta], \tilX_s, (X_s^{\theta'})') \big[\partial_x (X_s^{\theta', [\theta]})' \gamma'  +  \EE''[ \partial_x (X_s^{\theta'', [\theta]})'(\theta'') \gamma'' ] \big]\bigg\} \ud B^i_s.
	\end{align}
	
	We observe that this dynamic is identical to the dynamic for $\mc{D}X_t^\theta(\gamma)$ in \eqref{def:eq_DXtheta} and hence they are identical by uniqueness. Similarly, by using this result for $\mc{D}X_t^\theta(\gamma)$ and the same procedures, we are able to derive that $\EE''[\partial_\mu X_t^{x, [\theta]}(\theta'') \gamma'']$ is equal to $\mc{D} X_t^{x, [\theta]}(\gamma)$. So \eqref{def:eq_DXx[theta]} is proved. Moreover, $\partial_\mu X_t^{x, [\theta]}(v)$ exists and satisfies equation \eqref{def:eq_partialmuXx} by its definition.
	
	\item We first deduce the Malliavin derivative for $X^\theta$. Consider the Picard iteration given by
	\begin{align*}
	X_t^{\theta, 0} &= \theta, \\
	X_t^{\theta, k+1} &= \theta + \sum_{i=0}^d\int_0^t V_i(s, X_s^{\theta,k}, [\tilX^k_s], \tilX^k_s)\ud B^i_s,
	\end{align*}
	where $\tilX^k$ is a copy of $X^{\theta, k}$ independent of the Brownian motion and $\theta$.
	We have shown that such iteration induces a Cauchy sequence $\{\Phi^{(k)}(\Law(X_t^{\theta, 0})), k\in \N\}$ and a weak solution of the directed chain SDE. Since $V_0, V_i$ are bounded continuously differentiable, we have
	\begin{align*}
	\bm{D}^{\mc{H}, l}_r[ V_i^j(s, X_s^{\theta,k}, [\tilX^k_s], \tilX^k_s) ] = \partial V_i^j \bm{D}^{\mc{H}, l}_r X_s^{\theta, k}, 
	\end{align*}
	where we omit the arguments in $V_i$'s for notation simplicity. Note that $|\partial V_i^j |\le K$ for some constant $K>0$. We can then deduce that $V_i^j(s, X_s^{\theta,k}, [\tilX^k_s], \tilX^k_s) \in \bD^{1, \infty}$ by \cite[Proposition 1.5.5]{nualart2009malliavin}. Moreover, the Ito integral 
	$$\int_0^t V_i^j(s, X_s^{\theta,k}, [\tilX^k_s], \tilX^k_s) \ud B_s^i,  \quad i=1, \dots, d$$
	 belongs to $\bD^{1,2}$ and for $r\le t$, we have
	\begin{align*}
	\bm{D}^{\mc{H}, l}_r\Big[ \int_0^t V_i^j(s, X_s^{\theta,k}, [\tilX^k_s], \tilX^k_s) \ud B_s^i \Big] &= V_l^j(r, X_r^{\theta,k}, [\tilX^k_r], \tilX^k_r) \\
	&\hspace{2em}+ \int_r^t \bm{D}^{\mc{H},l}_r[V_i^j(s, X_s^{\theta,k}, [\tilX^k_s], \tilX^k_s)]\ud B_s^i.
	\end{align*}
	On the other hand, the Lebesgue integral $\int_0^t V_0^j(s, X_s^{\theta,k}, [\tilX^k_s], \tilX^k_s) \ud s$ is also in the space $\bD^{1,2}$ and have the dynamics
	$$\bm{D}^{\mc{H}, l}_r \Big[\int_0^t V_0^j(s, X_s^{\theta,k}, [\tilX^k_s], \tilX^k_s) \ud s\Big] = \int_0^t \bm{D}^{\mc{H}, l}_r[V_0^j(s, X_s^{\theta,k}, [\tilX^k_s], \tilX^k_s) ]\ud s.$$
	
	\hspace{2em}Therefore, the dynamic of $\bm{D}^{\mc{H}, l}_r[X_t^{\theta, k+1}]$ has exactly the form of \eqref{def:sde_DX} by the chain rule of Malliavin derivative. Due to the reason that $\tilX^k$ and $X^{\theta, k}$ has the same distribution, by Doob's maximal inequality and Burkholder's inequality,
	$$\EE[\sup_{0\le s\le t} | \bm{D}^{\mc{H}, l}_r X^{\theta, k}_s |^p] \le c_1,$$
	where $c_1$ is a constant that depends only on $K, d, p$ for $p\ge 2$. Moreover, we define a metric similar to \eqref{def:metric_process_measure} but raise the power to general $p\ge 1$, 
	$$D_{t, p}(\mu_1, \mu_2) :=\inf \bigg\{ \int(\sup_{0\le s\le t}|X_s(\omega_1) - X_s(\omega_2)|^p\wedge 1)\ud \mu(\omega_1, \omega_2) \bigg\}^{1/p}.$$
	Note that the weak convergence result of the $X^\theta$ holds under any metric $D_{t, p}$ with $p\ge 2$.
	
	We then have the following, 
	$$D_t(m^{k+1}, m^k)^2 \le c_1 \int_0^t D_{s, p}(\Law(X^{\theta,k}), \Law(X^{\theta, k-1})) \ud s + c_2 \int_0^t D_s(m^{k}, m^{k-1})^2 \ud s,$$
	by a similar approach as in the proof of Proposition \ref{prop:sol of directed chain sde}, where $c_1, c_2$ are positive constants depending on $K, d, p$ and $m^k = \Law(\bm{D}_r^{\mc{H}, l} X^{\theta, k})$. By iteration, we get that $\{m^k, k\in\N\}$ forms a Cauchy sequence and has limit.
	We have now proved that 
	\begin{align}
		\bm{D}^{\mc{H}}_r X_t^{\theta} &= \sigma\big(r, X_r^{\theta}, [X_r^\theta], \tilX_r \big)+ \sum_{i=0}^d \int_r^t \bigg(\partial V_i(s, X_s^{\theta}, [X_s^\theta], \tilX_s)\bm{D}^{\mc{H}}_r X_s^{\theta} \bigg) \ud B_s^i, 
	\end{align}
	and the solution of $\bm{D}^{\mc{H}}_r X_t^{\theta}$ exists uniquely in the weak sense. In the iteration, it can be easily proved by induction that $X^{\theta, k} \in \bD^{1, \infty}$ and the sequence $\bm{D}^{\mc{H}}_r X_t^{\theta, k}$ is uniformly bounded in $L^p(\Omega; H)$ for $p\ge 2$. Therefore, we have $X_t^\theta \in \bD^{1, \infty}$. The proof for $X^{x, [\theta]}_t$ is similar, we can set $X_t^{x, [\theta], 0} = \theta$ add another equation for $X^{x, [\theta], k}_t$ into the above Picard iteration
	$$X_t^{x, [\theta], k+1} = x + \sum_{i=0}^d\int_0^t V_i(s, X_s^{x, [\theta] ,k}, [\tilX^k_s], \tilX^{x,k}_s)\ud B^i_s.$$
	Then the procedures are the same as the deduction for $\bm{D}^{\mc{H}}_r X^{\theta}$.
\end{enumerate}
\end{proof}

For the purpose of more general applications, we want to make sure that the density for directed chain SDE is at least second order differentiable, hence we need to extend the above first order regularities to higher orders. Following \cite{crisan2018smoothing}, we provide a result for general case, which characterize $X_t^{x, [\theta]}$ as a Kusuoka-Stroock process.

\begin{thm}
\label{thm:kusuoka}
Suppose $V_0, \dots, V_d \in \mcC^{k,k,k}_{b,\Lip}([0,T]\times \RR^N\times \mcP_2(\RR^N)\times\RR^N; \RR^N)$, then $(t, x, [\theta])\mapsto X_t^{x, [\theta]} \in \bK_0^1(\RR^N, k)$. If, in addition, $V_0, \dots, V_d$ are uniformly bounded then $(t, x, [\theta]) \mapsto X_t^{x, [\theta]} \in \bK_0^0(\RR^N, k)$.
\end{thm}

Note that \cite[Proposition 6.7 and 6.8]{crisan2018smoothing} can be extended our directed chain case, since the coefficients $V_i:[0,T]\times \RR^N \times \mcP_2(\RR^N) \times \RR^N \to \RR^N$ in directed chain SDEs can be written as a map of the form $\Omega \times [0,T] \times \RR^N \times \mcP_2(\RR^N) \ni (\omega, t, x, \mu) \mapsto a(\omega, t, x, \mu)\in \RR^N$. This is because the auxiliary dependence on the neighborhood in the coefficients can be thought as the dependence on an initial state $x$, initial distribution $\mu$ and independent Brownian motions, which are implied in the term $\omega$. Moreover, we are able to take care of the extra term with $\mc{D}\tilX_s$ due to the differentiability and regularity of $V_i$.

Similar to Proposition \ref{prop:first-order derivatives}, each type of derivative (w.r.t. $x, \mu$ or $v$) of $X_t^{x, [\theta]}$ satisfies a linear equation. We will introduce a general linear equation, derive some a priori $L^p$ estimates on the solution and then show that this linear equation is again differentiable under some conditions in the next Lemma. Whenever we say $a_k$, $k=1,2,3$, we also mean $\tila_1$.

\begin{lm}
\label{lm:linear_Yx}
Let $v_r$ be one element of the tuple $\bm{v}=(v_1, \dots, v_{\#\bm{v}})$ and $Y^{x, [\theta]}(\bm{v})$ solve the following SDE
\begin{align}
Y_t^{x, [\theta]}(\bm{v}) &= a_0 + \sum_{i=0}^d \int_0^t \bigg\{ a_1^i(s, x, [\theta]) Y_s^{x, [\theta]}(\bm{v}) + \tila^i_1(s, x, [\theta])\tilY_s(\bm{v}) + a_2(s, x, [\theta], \bm{v})  \nonumber \\
&\hspace{2em}+ \EE'\big[ a_3^i(s, x, [\theta], \theta') (Y_s^{\theta', [\theta]})'(\bm{v}) + \sum_{r=1}^{\#\bm{v}}a_3^i(s, x, [\theta], \theta') (Y_s^{v_r, [\theta]})'(\bm{v})  \big] \bigg\}\ud B^i_s, \label{def:eq_linearderivative}
\end{align}
where, for all $i=1, \dots, d$, the coefficients $(t, x, [\theta], \bm{v})\mapsto a_k(t, x, [\theta], \bm{v})$ are continuously in $L^p(\Omega)$ $\forall p\ge 1$, $k=1,2,3$ and 
\begin{align*}
	& a_0 \in \RR^N, \\
	& a_1, \tila_1: \Omega\times [0,T] \times \RR^N \times \mcP_2(\RR^N) \to \RR^{N\times N} \\
	& a_2 : \Omega\times [0,T] \times \RR^N \times \mcP_2(\RR^N) \times (\RR^N)^{\#\bm{v}} \to \RR^N \\
	& a_3^i : \Omega' \times \Omega\times [0,T] \times \RR^N \times \mcP_2(\RR^N) \times \RR^N\to \RR^{N\times N}.
\end{align*} 
In \eqref{def:eq_linearderivative}, $(Y^{\theta', [\theta]})'$ is a copy of $Y^{x, [\theta]}$ on the probability space $(\Omega', \mcF', \PP')$ where the initial state is $\theta'$. Similarly, $(Y^{v_r, [\theta]})'$ is a copy of $Y^{x, [\theta]}$ on the probability space $(\Omega', \mcF', \PP')$ where the initial state is $v$. $\tilY$ is the neighborhood process, which has the same law as $Y^\theta$ and independent with Brownian motion $B$.
If we make the following boundedness assumptions
\begin{enumerate}
	\item $\sup_{x\in\RR^N, [\theta]\in \mcP_2(\RR^N), \bm{v}\in(\RR^N)^{\# \bm{v}}} \|a_2(\cdot, x, [\theta], \bm{v})\|_{\mcS_T^p} < \infty$,
	\item $a_1, \tila_1$ and $ a_3$ are uniformly bounded,
	\item $\sup_{x\in\RR^N, [\theta]\in \mcP_2(\RR^N), \bm{v}\in(\RR^N)^{\# \bm{v}}} \|a_2(\cdot, x, [\theta], \bm{v})\|_{\mcS_T^2} < \infty$.
\end{enumerate}
then we have the following estimate for $C = C(p, T, a_1, a_3)$
$$\|Y^{x, [\theta]}(\bm{v}) \|_{\mcS_T^p} \le C(|a_0| + \|a_2(\cdot, x, [\theta], \bm{v})\|_{\mcS_T^p} + \|a_2(\cdot, x, [\theta], \bm{v})\|_{\mcS_T^2}).$$
Moreover, we also get that the mapping 
$$[0,T] \times \RR^N \times \mcP_2(\RR^N) \times (\RR^N)^{\#\bm{v}} \ni (t, x, [\theta], \bm{v}) \mapsto Y_t^{x, [\theta]}(\bm{v}) \in L^p(\Omega)$$
is continuous.
\end{lm}

\begin{proof}
Note that $\|\tilY(\bm{v})\|_{\mcS^p_T} = \|(Y_s^{\theta', [\theta]})'(\bm{v}) \|_{\mcS^p_T}$ since they have the same distribution. The rest proof is identical to \cite[Lemma 6.7]{crisan2018smoothing} by using Gronwall's lemma and the  Burkholder-Davis-Gundy inequality a couple times.
\end{proof}

We now consider the differentiability of the generic process satisfying the linear equation in Lemma \ref{lm:linear_Yx}. To ease the burden on notation, we omit the $(t, x, [\theta])$ in $a_k$, and write $a_3\big|_{v=\theta'}$ to denote $a_k(s, x, [\theta], \theta')$ for instance.

\begin{prop}
\label{prop:differentiable_Yx}
Suppose that the process $Y^{x, [\theta]}(\bm(v))$ is as in Lemma \ref{lm:linear_Yx}. In addition to the assumptions of Lemma \ref{lm:linear_Yx}, we introduce the following differentiability assumptions:
\begin{enumerate}
\item[(a)] For $k=1,2,3$, all $(s, [\theta], \bm{v}) \in [0,T]\times \mcP_2(\RR^N)\times(\RR^N)^{\# \bm{v}}$ and each $p\ge 1$, $\RR^N \ni x\mapsto a_k(s, x, [\theta], \bm{v}) \in L^p(\Omega)$ is differentiable.
\item[(b)] For $k=1,2,3$, all $(s, [\theta], x) \in [0,T]\times \mcP_2(\RR^N)\times\RR^N$ and each $p\ge 1$, $(\RR^N)^{\# \bm{v}} \ni \bm{v}\mapsto a_k(s, x, [\theta], \bm{v}) \in L^p(\Omega)$ is differentiable.
\item[(c)] For all $(s, x, \bm{v}) \in [0,T]\times \RR^N \times(\RR^N)^{\# \bm{v}}$ the mapping $L^2(\Omega) \ni \theta \mapsto a_2(s, x, [\theta], \bm{v}) \in L^2(\Omega)$ is Fr\'echet differentiable.
\item[(d)] $a_k(s, x, [\theta], \bm{v}) \in \bD^{1, \infty}$ for $k=1,2,3$ and all $(s, x, [\theta], \bm{v}) \in [0,T]\times \mcP_2(\RR^N)\times(\RR^N)^{\# \bm{v}}$. Moreover, we assume the following estimates on the Malliavin derivatives hold.
$$\sup_{r\in[0,T]} \EE\bigg[ \sup_{s\in [0,T]} |\bm{D}^{\mc{H}}_r a_k(s, x, [\theta], \bm{v}) |^p \bigg] < \infty, \quad k=0,1,2,3.$$
\end{enumerate}
Then, for all $t\in [0,T]$ the following hold:
\begin{enumerate}
\item Under assumption (a), $x\mapsto Y_t^{x, [\theta]}(\bm{v})$ is differentiable in $L^p(\Omega)$ for all $p\ge 1$ and 
	$$ \partial_x Y_t^{x, [\theta]}(\bm{v}) :\overset{L^p}{=} \lim_{h\to 0} \frac{1}{|h|} \bigg( Y_t^{x+h, [\theta]}(\bm{v}) - Y_t^{x, [\theta]}(\bm{v}) \bigg), $$
	where the limit is taken in $L^p$ sense, satisfies
	\begin{align}
	\partial_x Y_t^{x, [\theta]}(\bm{v}) &= \sum_{i=0}^d \int_0^t \bigg\{ \partial_x a_1^i Y_s^{x, [\theta]}(\bm{v}) + a_1^i\partial_xY_s^{x, [\theta]}(\bm{v}) + \partial_x a_2^i \nonumber \\
	&\hspace{2em} + \EE'\bigg[ \partial_x a^i_3 \big|_{v=\theta'} (Y_s^{\theta', [\theta]})'(\bm{v}) + \sum_{r=1}^{\#\bm{v}} \partial_x a_3^i \big|_{v=v_r} (Y_s^{\theta', [\theta]})'(\bm{v}) \bigg] \bigg\}\ud B_s^i \nonumber
	\end{align}
\item Under assumption (b), $\bm{v} \mapsto Y_t^{x, [\theta]}(\bm{v})$ is differentiable in $L^p(\Omega)$ for all $p\ge 1$ and 
	$$ \partial_{\bm{v}} Y_t^{x, [\theta]}(\bm{v}) :\overset{{L^p}}{=} \lim_{h\to 0} \frac{1}{|h|} \bigg( Y_t^{x, [\theta]}(\bm{v}+h) - Y_t^{x, [\theta]}(\bm{v}) \bigg) $$
	satisfies
	\begin{align}
	\partial_{v_j} Y_t^{x, [\theta]}(\bm{v}) &= \sum_{i=0}^d \int_0^t \bigg\{ a_1^i \partial_{v_j} Y_s^{x, [\theta]}(\bm{v}) + \tila^i_1\partial_{v_j} \tilY_s(\bm{v}) + \partial_{v_j} a_2^i \nonumber \\
	&\hspace{2em} + \EE'\bigg[ \partial_{v} a_3^i \big|_{v=v_j} (Y_s^{v_j, [\theta]})'(\bm{v}) \bigg] + \EE'\bigg[ a^i_3\big|_{v=v_j} \partial_x (Y_s^{v_j, [\theta]})'(\bm{v})\nonumber \\
	&\hspace{2em}  + a_3^i\big|_{v=\theta'} \partial_{v_j} (Y_s^{\theta', [\theta]})'(\bm{v}) +\sum_{r=1}^{\#\bm{v}} a_3^i\big|_{v=v_r} \partial_{v_j} (Y_s^{v_r, [\theta]})'(\bm{v}) \bigg] \bigg\}\ud B_s^i. \nonumber
	\end{align}
\item Under assumption (a), (b) and (c), the maps $\theta \mapsto Y_t^{\theta, [\theta]}(\bm{v})$ and $\theta \mapsto Y_t^{x, [\theta]}(\bm{v})$ are Fr\'echet differentiable for all $(x, \bm{v}) \in \RR^N \times (\RR^N)^{\# \bm{v}}$, so $\partial_\mu Y_t^{x, [\theta]}(\bm{v})$ exists and it satisfies
\begin{align}
\partial_\mu Y_t^{x, [\theta]}(\bm{v}, \hat{v}) &= \sum_{i=0}^d \int_0^t \bigg\{ \partial_\mu a_1^i Y_s^{x, [\theta]}(\bm{v}) + a_1^i \partial_\mu Y_s^{x, [\theta]}(\bm{v}, \hat{v}) + \partial_\mu \tila_1^i \tilY_s(\bm{v}) + a_1^i \partial_\mu \tilY_s(\bm{v}, \hat{v})+ \partial_\mu a_2^i  \nonumber \\
&\hspace{2em} + \EE'\bigg[ \partial_\mu a_3^i (Y_s^{\theta', [\theta]})'(\bm{v}) + \partial_v a_3^i (Y_s^{\hat{v}, [\theta]})'(\bm{v}) + a_3^i \big|_{v=\theta'} \partial_\mu (Y_s^{\theta', [\theta]})'(\bm{v}, \hat{v}) \bigg] \nonumber \\
&\hspace{2em} + \EE'\bigg[ a_3^i \big|_{v=\hat{v}} \partial_x (Y_s^{\hat{v}, [\theta]})'(\bm{v}) + \sum_{r=1}^{\#\bm{v}} a_3^i \big|_{v=v_r}\partial_\mu (Y_s^{v_r, [\theta]})'(\bm{v}, \hat{v}) \bigg]\bigg\} \ud B_s^i. \nonumber
\end{align}

Moreover, we have the representation, for all $\gamma\in L^2(\Omega)$,
$$\mc{D}\bigg( Y_t^{\theta, [\theta]}(\bm{v}) \bigg)(\gamma) = \bigg( \partial_x Y_t^{x, [\theta]}(\bm{v}) \gamma + \EE''\big[ \partial_\mu Y_t^{x, [\theta]}(\bm{v}, \theta'') \gamma'' \big] \bigg)\bigg|_{x=\theta}.$$

\item Under assumption (d), $Y_t^{x, [\theta]} \in \bD^{1, \infty}$ and $\bm{D}^{\mc{H}}_r Y_t^{x, [\theta]}$ satisfies
\begin{align}
	\bm{D}^{\mc{H}}_r Y_t^{x, [\theta]}(\bm{v}) &=\bigg( a_1^j Y_r^{x, [\theta]} + \tila_1^j \tilY_r + a_2^j + \EE'\big[ a_3^j (Y_s^{x, [\theta]})'(\bm{v}) \big] \bigg)_{j=1,\dots,d} \nonumber \\
	&\hspace{2em} + \sum_{i=0}^d \int_0^t \bigg\{ \bm{D}^{\mc{H}}_r a_1^i Y_s^{x, [\theta]}(\bm{v}) +\bm{D}^{\mc{H}}_r \tila_1^i \tilY_s + a_1^i \bm{D}^{\mc{H}}_r Y_s^{x, [\theta]}(\bm{v}) + \tila_1^i \bm{D}^{\mc{H}}_r\tilY_s \nonumber \\
	&\hspace{2em} + \bm{D}^{\mc{H}}_r a_2^i + \EE'\big[  \bm{D}^{\mc{H}}_r a_3^i\big|_{v=\theta'} (Y_s^{x, [\theta]})'(\bm{v}) \big] \bigg\} \ud B^i_s. \nonumber
\end{align}
Moreover, the following bound holds:
\begin{equation}
	\sup_{r\le t} \EE\bigg[| \bm{D}^{\mc{H}}_r Y_t^{x, [\theta]}(\bm{v}) |^p \bigg] \le C \sup_{r\le t} \EE\bigg[ \sup_{r\le t\le T} \big(| \bm{D}^{\mc{H}}_r a_1 |^p +  | \bm{D}^{\mc{H}}_r \tila_1 |^p \big) \bigg]
\end{equation}
\end{enumerate}
\end{prop}

The limits in the above are taken in $L^p$ sense. When we say $k=1,2,3$ for the assumptions, we also means $\tila_1$.
\begin{proof}
See Proposition \ref{prop:first-order derivatives} and \cite[Proposition 6.8]{crisan2018smoothing} for the proof.
\end{proof}

We are now ready to prove Theorem \ref{thm:kusuoka}.
\begin{proof}[Proof of Theorem \ref{thm:kusuoka}]
The proof follows identically the proof of \cite[Theorem 3.2]{crisan2018smoothing}, where we apply Lemma \ref{lm:linear_Yx} and Proposition \ref{prop:differentiable_Yx}.
\end{proof}

\subsection{Integration by Parts Formulae} \label{sec: IBP}

Now we introduce some operators acting on the Kusuoka-Stroock processes. These will be used later in the integration by parts formulae. We first make the following common uniform ellipticity assumption.

\begin{assump}[Uniform Ellipticity]
\label{assump:ue}
Let $\sigma: [0,T]\times \RR^N \times \mcP_2(\RR^N) \times \RR^N \to \RR^{N\times d}$ be given by
$$\sigma(t, z, \mu, \tilde{z}) := [V_1(t, z, \mu, \tilde{z}), \dots, V_d(t, z, \mu, \tilde{z})].$$
We assume that there exists $\epsilon>0$ such that, for all $\xi \in \RR^N$, $z\in\RR^N$ and $\mu\in\mcP_2(\RR^N)$,
$$\xi^\top \sigma(t, z, \mu, \tilde{z})\sigma(t, z, \mu, \tilde{z})^\top \xi \ge \epsilon |\xi|^2. $$
\end{assump}

For a function $\Psi: [0,T]\times \RR^N\times \mcP_2(\RR^N) \to \bD^{n, \infty}$, the following operators acting on Kusuoka-Stroock processes in $\bK_r^q(\RR, n)$ with multi-index $\alpha=(i)$ and $(t, x, [\theta]) \in [0,T]\times \RR^N\times \mcP_2(\RR^N)$ are given by
\begin{align*}
I_{(i)}^1(\Psi)(t, x, [\theta]) &:= \frac{1}{\sqrt{t}} \delta_{\mc{H}}\bigg(r \mapsto \Psi(t, x, [\theta])\big( \sigma^\top \big( \sigma \sigma^\top \big)^{-1} (r, X_r^{x, \mu}, [X_r^{\theta}], \tilX_r)\partial_x X_r^{x, \mu} \big)_i \bigg) \nonumber \\
I_{(i)}^2(\Psi)(t, x, [\theta]) &:= \sum_{j=1}^N I^1_{(j)} \bigg( \big( \partial_x X_t^{x, \mu} \big)^{-1}_{j,i} \Psi(t, x, [\theta]) \bigg), \\
I_{(i)}^3(\Psi)(t, x, [\theta]) &:= I_{(i)}^1(\Psi)(t, x, [\theta]) + \sqrt{t}\partial^i\Psi(t, x, [\theta]), \\
\mcI_{(i)}^1(\Psi)(t, x, [\theta], v_1) &:= \frac{1}{\sqrt{t}}\delta_{\mc{H}}\bigg( r\mapsto \big(\sigma^\top \big(\sigma \sigma^\top\big)^{-1}(r, X_r^{x, \mu}, [X_r^\theta], \tilX_r) \\
&\hspace{4em} \partial_xX_r^{x, \mu} (\partial_x X_t^{x, \mu})^{-1} \partial_\mu X_t^{x, [\theta]}(v_1) \big)_{i} \Psi(t, x, [\theta]) \bigg), \\
\mcI_{(i)}^3(\Psi)(t, x, [\theta], v_1) &:= \mcI_{(i)}^1(\Psi)(t, x, [\theta], v_1) + \sqrt{t}(\partial_\mu \Psi)_i(t, x, [\theta], v_1).
\end{align*}

For a general multi-index $\alpha=(\alpha_1, \dots, \alpha_n)$, we inductively define
$$I_\alpha^1 := I^1_{\alpha_n} \circ I^1_{\alpha_{n-1}} \circ \cdots \circ I^1_{\alpha_1},$$
the definition of the other operators are analogue to $I^1_\alpha$. The following Proposition follows directly from our previous discussion and the definition of the Kusuoka-Stroock process.

\begin{prop}
\label{prop:general_kusuoka}
If $V_0,\dots, V_d \in \mcC^{k,k,k}_{b, \Lip}([0,T]\times \RR^N \times \mcP_2(\RR^N) \times \RR^N; \RR^N)$, Assumption \ref{assump:ue} holds and $\Psi\in \bK_r^q(\RR, n)$, then $I_\alpha^1(\Psi)$ and $I_\alpha^3(\Psi)$, are all well-defined for $|\alpha|\le (k\wedge n)$. $I_\alpha^2(\Psi), \mcI^1_\alpha(\Psi)$ and $\mcI^3_\alpha(\Psi)$ are well defined for $|\alpha|\le n\wedge (k-2)$. Moreover, 
\begin{align*}
	I_\alpha^1(\Psi), I_\alpha^3(\Psi) &\in \bK_r^{q+2|\alpha|}(\RR, (k\wedge n) -|\alpha|), \\
	I_\alpha^2(\Psi) &\in \bK_r^{q+3|\alpha|}(\RR, [n \wedge (k-2)] -|\alpha|), \\
	\mcI_\alpha^1(\Psi), \mcI_\alpha^3(\Psi) &\in \bK_r^{q+4|\alpha|}(\RR,  [n \wedge (k-2)] -|\alpha|).
\end{align*}
If $\Psi\in\bK_r^0(\RR, n)$ and $V_0, \dots, V_d$ are uniformly bounded, then
\begin{align*}
	I_\alpha^1(\Psi), I_\alpha^3(\Psi) &\in \bK_r^0(\RR, (k\wedge n) -|\alpha|), \\
	I_\alpha^2(\Psi) &\in \bK_r^{0}(\RR, [n \wedge (k-2)] -|\alpha|), \\
	\mcI_\alpha^1(\Psi), \mcI_\alpha^3(\Psi) &\in \bK_r^{0}(\RR,  [n \wedge (k-2)] -|\alpha|).
\end{align*}
\end{prop}

From now on, the Integration by Parts Formulae (IBPF) follow in the same way as \cite[Sec 4.]{crisan2018smoothing} by replacing $\bm{D}, \delta$ by $\bm{D}^{\mc{H}}, \delta_{\mc{H}}$ and using integral by parts for this partial Malliavin derivative.

Integration by parts formulae in the space variable are established in the following Proposition.
\begin{prop}[Proposition 4.1, \cite{crisan2018smoothing}]
\label{prop:IBPF_space}
Let $f\in\mc{C}_b^\infty(\RR^N, \RR)$ and $\Psi \in \bK^q_r(\RR, n)$, then
\begin{enumerate}
    \item[1.] If $|\alpha| \le n\wedge k$, then
    $$\EE\big[ \partial_x^\alpha\big( f\big( X_t^{x, [\theta]} \big) \big)\Psi(t, x, [\theta]) \big] = t^{-|\alpha|/2} \EE\big[ f\big( X_t^{x, [\theta]}\big) I_\alpha^1(\Psi)(t, x, [\theta])  \big].$$
    
    \item[2.] If $|\alpha| \le n\wedge (k-2)$, then
    $$ \EE\big[ (\partial^\alpha f) \big( X_t^{x, [\theta]} \big)\Psi(t, x, [\theta]) \big] = t^{-|\alpha|/2} \EE\big[ f\big( X_t^{x, [\theta]}\big) I_\alpha^2(\Psi)(t, x, [\theta])  \big]. $$
    
    \item[3.] If $|\alpha| \le n\wedge k$, then
    $$\partial_x^\alpha \EE\big[ f\big( X_t^{x, [\theta]} \big)\Psi(t, x, [\theta]) \big] = t^{-|\alpha|/2} \EE\big[ f\big( X_t^{x, [\theta]}\big) I_\alpha^3(\Psi)(t, x, [\theta])  \big].$$ 
    
    \item[4.] If $|\alpha| + |\beta| \le n \wedge (k-2)$, then
    $$ \partial_x^\alpha\EE\big[ (\partial^\beta f) \big( X_t^{x, [\theta]} \big)\Psi(t, x, [\theta]) \big] = t^{-(|\alpha|+|\beta|)/2} \EE\big[ f\big( X_t^{x, [\theta]}\big)  I_\alpha^3\big((I^2_\beta\Psi)\big)(t, x, [\theta])  \big]. $$
\end{enumerate}
\end{prop}

\begin{proof}
\begin{enumerate}
    \item[1.]  First, we note that Equation \eqref{def:first_order_x} satisfied by $\partial_x X_t^{x, [\theta]}$ and Equation \eqref{def:sde_DX}
    satisfied by $\bm{D}^{\mc{H}}_r X_t^{x, [\theta]}$ are the same except their initial condition. It therefore follows from our discussion of partial Malliavin derivative that
\begin{equation}
\label{def:eq_first_order_and_malliavin}
\partial_x X_t^{x, [\theta]} = \bm{D}^{\mc{H}}_r X_t^{x, [\theta]} \sigma^\top\big(\sigma \sigma^\top\big)^{-1}(r, X_r^{x, [\theta]}, [X^\theta_r], \tilX_r) \partial_x X_r^{x, [\theta]}.
\end{equation}

We are then allowed to compute the followings for $f\in \mc{C}^\infty_b(\RR^N, \RR)$,
\begin{align*}
    \EE\big[ \partial_x\big( f\big( X_t^{x, [\theta]} \big) \big)\Psi(t, x, [\theta]) \big] &=
    \EE\big[ \partial f\big( X_t^{x, [\theta]} \big) \partial_x X_t^{x, [\theta]} \Psi(t, x, [\theta]) \big] \\
    &= \frac{1}{t}\EE\bigg[\int_0^t \partial f\big( X_t^{x, [\theta]} \big) \partial_x X_t^{x, [\theta]} \Psi(t, x, [\theta])\ud r \bigg] \\
    &= \frac{1}{t}\EE\bigg[\int_0^t \partial f\big( X_t^{x, [\theta]} \big) \bm{D}^{\mc{H}}_r X_t^{x, [\theta]} \sigma^\top\big(\sigma \sigma^\top\big)^{-1}(r, X_r^{x, [\theta]}, [X^\theta_r], \tilX_r)\\
    &\hspace{4em} \times \partial_x X_r^{x, [\theta]} \Psi(t, x, [\theta])\ud r \bigg]\\
    &= \frac{1}{t}\EE\bigg[\int_0^t  \bm{D}^{\mc{H}}_r f\big( X_t^{x, [\theta]} \big) \sigma^\top\big(\sigma \sigma^\top\big)^{-1}(r, X_r^{x, [\theta]}, [X^\theta_r], \tilX_r)\\
    &\hspace{4em}\times \partial_x X_r^{x, [\theta]} \Psi(t, x, [\theta])\ud r \bigg]\\
    &= \frac{1}{t}\EE\bigg[ f\big( X_t^{x, [\theta]} \big) \delta_{\mc{H}}\bigg( r \mapsto \Psi(t, x, [\theta]) \\
    &\hspace{4em}\times \big( \sigma^\top \big( \sigma \sigma^\top \big)^{-1} (r, X_r^{x, \mu}, [X_r^{\theta}], \tilX_r)\partial_x X_r^{x, \mu} \big)  \bigg)
    \bigg],
\end{align*}
where we have applied partial Malliavin calculus integration by parts from Equation \eqref{def:eq_partialMalliavin_IBPF} in the last equality. This proves the result for $|\alpha| = 1$. By Proposition \ref{prop:general_kusuoka}, $I_\alpha^1(\Psi) \in \bK_r^{q+2} (\RR, (k\wedge n) - 1)$ when $|\alpha|=1$. We can then repeat the above procedures iteratively to get to desired result.

\item[2.] By the chain rule, 
\begin{align*}
    \EE\big[ (\partial^i f) \big( X_t^{x, [\theta]} \big)\Psi(t, x, [\theta]) \big] &= \sum_{j=1}^N \EE\bigg[ \partial_{x_i}\big(f\big(X_t^{x, [\theta]}\big) \big)\bigg( \big( X_t^{x, [\theta]} \big)^{-1} \bigg)^{j,i}\Psi(t, x, [\theta]) \bigg] \\
    &= t^{-1/2} \sum_{j=1}^N\EE\bigg[ f\big(X_t^{x, [\theta]}\big) I^1_{(j)} \bigg(\bigg( \big( X_t^{x, [\theta]} \big)^{-1} \bigg)^{j,i}\Psi(t, x, [\theta])  \bigg) \bigg] \\
    &= t^{1/2} \EE\big[f\big(X_t^{x, [\theta]}\big) I^2_{(i)}(\Psi)(t, x, [\theta])\big],
\end{align*}
where we apply the result in part 1 to the second equality. From Proposition \ref{prop:general_kusuoka}, $I_{(i)}^2(\Psi) \in \bK^{q+3}_r(\RR, (n\wedge(k-2))-1)$, so since $|\alpha|\le (n\wedge(k-2))$, the proof follows from applying the same arguments for another $|\alpha|-1$ times.

\item[3.] By part 1 and direct computation,
\begin{align*}
    \partial_x^i \EE\big[ f\big( X_t^{x, [\theta]} \big)\Psi(t, x, [\theta]) \big] &= \EE\big[ \partial_x^i f\big( X_t^{x, [\theta]} \big)\Psi(t, x, [\theta]) + f\big( X_t^{x, [\theta]} \big)\partial_x^i \Psi(t, x, [\theta]) \big]\\
    &= t^{-1/2} \EE\bigg[f\big( X_t^{x, [\theta]} \big) \big\{I^1_{i}(\Psi)(t,x,[\theta])+\sqrt{t}\partial_x^i \Psi(t, x, [\theta])  \big\} \bigg],
\end{align*}
which proves the result for $|\alpha|=1$. Again, we have $I^3_{\alpha}(\Psi) \in \bK_r^{q+2}(\RR, (k\wedge n)-1)$ when $|\alpha| =1$. Then the proof follows from iterative implementation of the above procedure.

\item[4.] This part follows from parts 2 and 3 directly.
\end{enumerate}
\end{proof}

Similar to the integration by parts in space variable, we can also derive integration by parts in the measure variable as follows.

\begin{prop}[Proposition 4.2, \cite{crisan2018smoothing}]
\label{prop:IBPF_measure}
Let $f\in\mc{C}_b^\infty(\RR^N, \RR)$ and $\Psi \in \bK^q_r(\RR, n)$, then
\begin{enumerate}
    \item[1.] If $|\beta| \le n\wedge (k-2)$, then
    $$\EE\big[ \partial_\mu^\beta\big( f\big( X_t^{x, [\theta]} \big) \big)(\bm{v})\Psi(t, x, [\theta]) \big] = t^{-|\beta|/2} \EE\big[ f\big( X_t^{x, [\theta]}\big) \mc{I}_\beta^1(\Psi)(t, x, [\theta],\bm{v}) \big].$$
    
    \item[2.] If $|\beta| \le n\wedge (k-2)$, then
    $$\partial_\mu^\beta \EE\big[ f\big( X_t^{x, [\theta]} \big)\Psi(t, x, [\theta]) \big](\bm{v}) = t^{-|\beta|/2} \EE\big[ f\big( X_t^{x, [\theta]}\big) \mc{I}_\beta^3(\Psi)(t, x, [\theta], \bm{v})  \big].$$ 
    
    \item[3.] If $|\alpha|+|\beta|\le n\wedge (k-2)$, then
    $$\partial_\mu^\beta \EE\big[ (\partial^\alpha f) \big( X_t^{x, [\theta]} \big) \Psi(t,x,[\theta]) \big](\bm{v}) = t^{-(|\alpha|+|\beta|)/2} \EE\bigg[ f\big( X_t^{x, [\theta]} \big) \mc{I}^3_\beta\big( I^2_\alpha(\Psi) \big)(t,x,[\theta], \bm{v}) \bigg].$$
    
\end{enumerate}
\end{prop}

\begin{proof}
The proofs use the same idea as Proposition \ref{prop:IBPF_space} and the Equation \eqref{def:eq_first_order_and_malliavin}.
\end{proof}

We now consider the integration by parts formulae for the derivatives of the mapping:
$$x\mapsto \EE[f(X_t^{x,\delta_x})].$$
Let us introduce the following operator acting on $\mc{K}_r^q(\RR, M)$, the set of the Kusuoka-Stroock processes do not depend on measure $\mu$. For $\alpha=(i)$,
$$J_{(i)} (\Phi)(t,x):= I_{(i)}^3(\Phi)(t,x,\delta_x) + \mc{I}_{(i)}^3(\Phi)(t,x,\delta_x)$$
and for $\alpha=(\alpha_1, \dots, \alpha_n)$,
$J_\alpha (\Phi) := J_{\alpha_n}\circ \dots \circ J_{\alpha_1}$.


\begin{thm}
\label{thm:IBPF_map}
Let $f\in \mc{C}^\infty_b(\RR^N;\RR)$. For all multi-indices $\alpha$ on $\{1,\dots,N\}$ with $|\alpha|\le k-2$,
$$\partial_x^\alpha \EE\big[ f\big( X_t^{x,\delta_x} \big) \big] = t^{-|\alpha|/2} \EE\big[ f\big( X_t^{x,\delta_x} \big) J_\alpha(1)(t,x) \big].$$
In particular, we get the following bound,
$$\big| \partial_x^\alpha \EE\big[ f\big( X_t^{x,\delta_x} \big) \big] \big| \le C \|f\|_\infty t^{-|\alpha|/2}(1+|x|)^{4|\alpha|}$$
\end{thm}
\begin{proof}
Since $\delta_x$ depends on $x$, we have
$$\partial_x^i \EE\big[f\big( X_t^{x,\delta_x} \big)\big]= \partial_z^i \EE\big[ f\big( X_t^{x,\delta_x} \big) \big]\big|_{z=x} + \partial_\mu^i \EE\big[ f\big( X_t^{x,[\theta]} \big) \big](v)\big|_{[\theta]=\delta_x, v=x}, $$
then for $|\alpha|=1$ the result yields by Proposition \ref{prop:IBPF_space} and \ref{prop:IBPF_measure}. The proof is completed by repeating this procedure for another $|\alpha|-1$ times. 
\end{proof}

The following Corollary is useful for the smoothness of densities of directed chain SDEs.
\begin{cor}
\label{cor:IBPF_smoothdensity}
Let $f\in \mc{C}^\infty_b(\RR^N; \RR)$, $\alpha$ and $\beta$ are multi-indices on $\{1,\dots,N\}$ with $|\alpha|+|\beta|\le k-2$. Then,
$$\partial_x^\alpha \EE\big[ (\partial^\beta f)\big( X_t^{x,\delta_x} \big) \big] = t^{-\frac{|\alpha|+|\beta|}{2}}\EE\big[ f\big( X_t^{x,\delta_x} \big) I_\beta^2(J_\alpha(1))(t,x) \big]$$
and $I_\beta^2(J_\alpha(1)) \in \mc{K}_0^{4|\alpha|+3|\beta|}(\RR, k-2-|\alpha|-|\beta|)$.
\end{cor}

\begin{proof}
The proof follows from Theorem \ref{thm:IBPF_map} and Proposition \ref{prop:IBPF_space}.
\end{proof}

\subsection{Smooth Densities}
We are now ready to prove the main theorem of this section.
\begin{thm}
\label{thm:smooth_density}
We assume Assumption \ref{assump:ue} holds and $V_0, \dots, V_d\in \mcC^{k,k,k}_{b, \Lip}([0,T]\times\RR^N\times \mcP_2(\RR^N)\times \RR^N; \RR^N)$. Let $\alpha, \beta$ be multi-indices on $\{1, \dots, N\}$ and let $k\ge |\alpha| + |\beta| +N+2$. Assume also the initial state for directed chain SDE is $\theta\equiv x$, i.e. $[\theta]=\delta_x$. Then the directed chain SDE \eqref{def:sde_random_coef} coincides with the alternative SDE \eqref{def:sde_random_coef_x}. For all $t\in[0,T]$, $X_t^{x, \delta_x}$ has a density $p(t, x, \cdot)$ such that $(x, z)\mapsto \partial_x^\alpha \partial_z^\beta p(t, x, z)$ exists and is continuous. Moreover, there exists a constant $C$ which depends on $T, N$ and bounds on the coefficients, such that for all $t\in (0,T]$
\begin{equation}
    |\partial_x^\alpha \partial_z^\beta p(t, x, z)| \le C(1 + |x|)^{4|\alpha| + 3|\beta|+3N} t^{-\frac{1}{2}(N+|\alpha| + |\beta|)},\,\,\,\, x\in \RR^N,\, z \in \RR^N. \label{def:density_upper_bound}
\end{equation}

If $V_0, \dots, V_d$ are bounded, then the following estimate holds, for all $t\in(0,T]$
$$|\partial_x^\alpha \partial_z^\beta p(t, x, z)| \le C t^{-\frac{1}{2}(N+|\alpha| + |\beta|)} \exp\big( -C\frac{|z-x|^2}{t} \big),\,\,\,\, x\in \RR^N,\, z \in \RR^N.$$
\end{thm}

\begin{proof}
The proof is verbatim to Theorem 6.1 of \cite{crisan2018smoothing} by applying our integration by parts formulae established in Corollary \ref{cor:IBPF_smoothdensity} and Lemma 3.1 in \cite{taniguchi1985}.
\end{proof}

Theorem \ref{thm:smooth_density} presents the smoothness result for $X_t^{x, \delta_x}$ and it can be generalized to $X_t^\theta$ with an general initial distribution $[\theta]$.

\begin{cor}
\label{cor:smooth_density_with_general_initial}
Suppose Assumption \ref{assump:ue} holds and $V_0, \dots, V_d\in \mcC^{k,k,k}_{b, \Lip}([0,T]\times\RR^N\times \mcP_2(\RR^N)\times \RR^N; \RR^N)$. Let $\theta$ be a random variable in $\RR^N$ with finite moments of all orders. For any multi-index $\beta$ on $\{1, \dots, N\}$ such that $k\ge |\beta| +N+2$, we have that for all $t\in[0,T]$, $X_t^{\theta}$ has a density $p_\theta(t, \cdot)$ such that $z\mapsto \partial_z^\beta p_\theta(t, z)$ exists and is continuous.
\end{cor}
\begin{proof}
    The proof is done by taking expectation on both sides of the inequality \eqref{def:density_upper_bound} with respect to the initial distribution $\theta$ and applying dominated convergence theorem, where we use the assumption that $\theta$ has finite moments. 
\end{proof}

The above existence and smoothness results on the marginal density of a single object can be extended to the joint distribution for any number of adjacent particles. Namely, for a fixed integer $m \ge 1$, we may construct the system of stochastic processes $(\widetilde{X}^{0}_{\cdot}, \widetilde{X}^1_{\cdot}, \widetilde{X}^2_{\cdot},\dots, \widetilde{X}^m_{\cdot})$ such that  $(\widetilde{X}^{m-1}_{\cdot}, \widetilde{X}^m_{\cdot})\equiv (X^\theta_\cdot, \widetilde{X}_{\cdot})$ in  \eqref{eq:1.1}, and $\widetilde{X}^i_{\cdot}$ depends on the adjacent process $\widetilde{X}^{i+1}_{\cdot}$ and Brownian motion $\widetilde{B}^i_\cdot$, independent of $\widetilde{X}^{i+1}$, in the same fashion as of $(X^{\theta}_{\cdot}, \widetilde{X}_{\cdot})$ in \eqref{eq:1.1} for $i = 0, \ldots , m-1$. 

\begin{cor} \label{cor:coupled_density}
Suppose Assumption \ref{assump:ue} holds and $V_0, \dots, V_d\in \mcC^{k,k,k}_{b, \Lip}([0,T]\times\RR^N\times \mcP_2(\RR^N)\times \RR^N; \RR^N)$ and $\theta$ has finite moments. Then the joint density of the process $(X^{\theta}_{\cdot}, \widetilde{X}^1_{\cdot}, \widetilde{X}^2_{\cdot},\dots, \widetilde{X}^m_{\cdot})$ exists and is continuous at any $t\in[0,T]$, where $\widetilde{X}^1_{\cdot}\equiv \widetilde{X}_{\cdot}$ and $\widetilde{X}^i_{\cdot}$ depends on $\widetilde{X}^{i+1}_{\cdot}$ in the same fashion as of $(X^{\theta}_{\cdot}, \widetilde{X}_{\cdot})$ in \eqref{eq:1.1}.
\end{cor}
\begin{proof}
    We consider the process evolving in space $\RR^{(m+1)N}$ defined by 
    $$Y_{\cdot}:=(\widetilde{X}^{0}_{\cdot}, \widetilde{X}^1_{\cdot}, \widetilde{X}^2_{\cdot},\dots, \widetilde{X}^m_{\cdot})$$
     and the neighborhood process $\widetilde{Y}_{\cdot}:=(\widetilde{X}^{m+1}_{\cdot}, \widetilde{X}^{m+2}_{\cdot},\dots, \widetilde{X}^{2m+1}_{\cdot})$. Now $(Y_\cdot, \widetilde{Y}_{\cdot})$ satisfies the directed chain structure and it can be proved that this new directed chain SDE structure $Y_\cdot$ also satisfies Assumption \ref{assump:ue}. Hence the existence and continuity follow from Theorem \ref{thm:smooth_density} and Corollary \ref{cor:smooth_density_with_general_initial}. In particular, if $m=1$, the coupled process $Y_{\cdot}$ is defined by
    $$Y_{t} = Y_{0} + \sum_{i=1}^{2d}\int_0^T V^y_i(s, Y_{s}, \Law(Y_{s}), \widetilde{Y}_{s}) \ud B_{s}^{y, i} ,$$
where the diffusion coefficients $V^y_i$, $i = 1, \ldots , 2d$ are given by  
\begin{align*}
     V^y_i := \left\{ \begin{array}{cc}
        \big(V_i(s, X_{s}, \Law(X_{s}), \widetilde{X}^1_{s}), \mathbf{0}\big)^T \in \RR^{2N},  & i=1,\dots d, \\
        \big(\mathbf{0}, V_{i-d}(s, \tilX^1_{s}, \Law(\tilX^1_{s}), \tilX^{2}_{s})\big)^T\in \RR^{2N},  & i=d+1,\dots 2d,
     \end{array}  \right.
\end{align*}
$B^y$ is independent standard Brownian motions in $\RR^{2d}$ and $\mathbf{0} \in \RR^N$ is a zero vector. 
\end{proof}

\subsection{Markov Random Fields}
The existence of density in Theorem \ref{thm:smooth_density} is closely related to the local Markov property (or Markov random fields) of the directed chain structure. Here, we shall elaborate the relation briefly. A similar topic has been studied by \cite{lacker2021locally} on the undirected graph with locally interacting only on the drift terms. Their approach is to apply a change of measure under which the diffusion coefficients at one vertex of the undirected graph do not depend on the diffusions at the other vertexes of the graph, in order to get the factorization of the probability measure. Usually, this Markov property is only discussed for the undirected graph or directed acyclic graph. The finite particle system that approximates the directed chain structure discussed in \cite{DETERING20202519} admits a loop structure in the finite graph. More precisely, the finite system of $n$ particles $(X_{1,\cdot}^{(n)}, \ldots , X_{n,\cdot}^{(n)})$ is constructed in a loop of size $n$ so that $X_{1,\cdot}^{(n)}$ depends on $X_{2,\cdot}^{(n)}$, $X_{2,\cdot}^{(n)}$ depends on $X_{3,\cdot}^{(n)}$, $\dots$, $X_{n-1,\cdot}^{(n)}$ depends on $X_{n,\cdot}^{(n)}$ and $X_{n, \cdot}^{(n)}$ depends on $X_{1,\cdot}$.
However, when the size $n$ of this loop is forced to be infinity, i.e., $n\to \infty$, we can then treat the dependence of the system on any finite subgraph as the system on an acyclic graph \cite[Section 3]{DETERING20202519}, as \eqref{def:sde_X} in our paper. An illustration is given in Figure \ref{fig:directedchain}.
\begin{figure}
    \centering
    \includegraphics[width=0.8\textwidth]{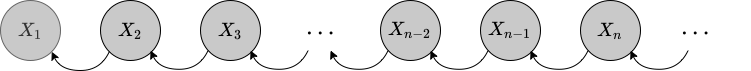}
    \caption{This figure shows a finite cut of the infinite directed chain, i.e., $X_k$ is affected by $X_{k+1}$.}
    \label{fig:directedchain}
\end{figure}

\begin{prop}\label{prop:localMarkov}
The directed chain SDEs described in \eqref{def:sde_X} form a first order Markov random fields, or we say it has the local Markov property.
\end{prop}

We follow the notations and terminology in \cite{lauritzen1996graphical}. Given a directed graph $G = (V, E)$ with vertexes $V$ and edges $E$, for a vertex $\nu\in V$, let $\mathcal{X}_\nu$ denote the generic space of vertex $\nu$ and $\pa(\nu)\in V$ denote all of its parents. In the infinite directed chain case, $\pa(X_{k,\cdot})=X_{k+1,\cdot}$.

\begin{defn}[Recursive Factorization]
Given a directed graph $G = (V, E)$, we say the probability distribution $P^G$ admits a recursive factorization according to $G$, if there exists non-negative functions, henceforth referred to as kernels, $k^{\nu}(\cdot, \cdot), \nu\in V$ defined on $\mathcal{X}_\nu \times \mathcal{X}_{\pa(\nu)}$, such that
$$\int k^{\nu}(y_\nu, x_{\pa(\nu)})\mu_{\nu}(\ud y_{\nu}) = 1$$   
and $P^G$ has density $f^G$ with respect to a product measure $\mu$, which is defined on the product space $\prod_{\nu \in V} \mathcal{X}_\nu$ by $\mu_\nu$ a measure defined on each  $\mathcal{X}_\nu$, where 
$$f^G(x) = \prod_{\nu\in V}k^{\nu}(x_\nu, x_{\pa(\nu)}).$$
\end{defn}

\begin{proof}[Proof of Proposition \ref{prop:localMarkov}]
Thanks to the special chain-like structure, it can be shown that the distribution of the chain satisfies the \textit{recursive factorization} property, where the existence and continuity of the kernel functions are given by Theorem~\ref{thm:smooth_density} and Corollary \ref{cor:smooth_density_with_general_initial}. For it, on a filtered probability space, let us consider a system of the directed chain diffusion $X_{i,t} $, $i \in \mathbb N$, $t \ge 0$ on the infinitely graph with the vertexes $\mathbb N = \{1, 2, \ldots \}$. Firstly, the coupled diffusion $(X_{1,\cdot}, X_{2,\cdot})\equiv (X^\theta_{\cdot}, \widetilde{X}_{\cdot}) $ satisfy the directed chain stochastic equation and have a continuous density by Corollary \ref{cor:coupled_density} and we denote this joint density by $g(\cdot, \cdot):\RR^N\times \RR^N \to \RR$. We then build the chain recursively by the following rule: given $X_{k, \cdot}$, initial state $X_{k+1, 0}$ and Brownian motion $B_{k+1, \cdot}$ independent of $(X_{1,\cdot}, \dots, X_{k,\cdot}, X_{k+1, 0})$, we construct $X_{k+1, \cdot}$ according to the distribution of $(X^\theta_{\cdot}, \widetilde{X}_{\cdot})$. 

Defining the kernel functions in the following way
$$k^\nu(x_\nu, x_{\pa(\nu)}) :=\left\{ \begin{array}{cc}
    g(x_\nu, x_{\pa(\nu)}), & \text{if } \nu = X_{1}, \\ 
    ({g(x_{\nu})})^{-1}{g(x_\nu, x_{\pa(\nu)})}, & \text{if } \nu = X_{k}, \,k\ge 2,
\end{array} \right.$$
i.e., the conditional density of $X_{k+1, t}$ given $X_{k, t}$ for $k\ge 2$, proves the \textit{recursive factorization} property of the chain on any finite cut $(X_{1,t}, X_{2,t},\dots, X_{m,t}), \, \forall m\in\N$ of the infinite chain for any $t\in[0,T]$, as well as the local Markov property following from \cite[Theorem 3.27]{lauritzen1996graphical}, which is also called the first order Markov random field in the context of \cite{lacker2021locally}. This result can also be verified by a  filtering problem build upon this directed chains structure that we omit due to the page limitation.
\end{proof}

\subsection{Relation to PDE}
We have constructed the integration by parts formulae to argue that the density of directed chain SDEs is smooth in section \ref{sec: IBP}, which is also the tool for constructing solutions to a related PDE problem. To ease notation, we will omit the time dependences in coefficients of SDEs through this section, i.e. we will write $V(X_t^{x, [\theta]}, [X_t^\theta], \tilX_t):=V(t, X_t^{x, [\theta]}, [X_t^\theta], \tilX_t)$. In particular, we are interested in the function 
$U(t, x, [\theta]) = \EE[g(X_t^{x, [\theta]}, [X_t^\theta])]$, $t \in [0, T]$, $x \in \mathbb R^d$ for some sufficiently function $g$. Here $X_\cdot^\theta$ is the solution of \eqref{def:sde_random_coef}-\eqref{def:sde_random_coef_constrain} with random initial $\theta$ and $X_t^{x, [\theta]}$ is the solution to \eqref{def:sde_random_coef_x} with deterministic initial $x$. They depend on a neighborhood process $\tilX_\cdot$ with an independent initial random vector $\tilde{\theta}$. Recall the flow property ~\eqref{def:flow} in section \ref{sec:flow}. It follows that for every $0 \le t \le t + h \le T$, $x \in \mathbb R^d$, 
$$U(t+h, x, [\theta]) = \mathbb E [ g(X_{t+h}^{x, [\theta]}, [X^\theta_{t+h}])] =  \EE\big[ U(t, X_h^{x, [\theta]}, [X_h^\theta]) \big].$$
Hence 
\begin{align}
\label{def:diff_U}
    U(t+h, x, [\theta]) - U(t, x, [\theta]) &= U(t, x, [X_h^\theta])-U(t, x, [\theta])  + \EE\big[ U(t, X_h^{x, [\theta]}, [X_h^\theta]) -  U(t, x, [X_h^\theta])\big] \nonumber \\
    &= I - \EE[J],
\end{align}
where we define $I = U(t, x, [X_h^\theta])-U(t, x, [\theta])$ and $J = U(t, X_h^{x, [\theta]}, [X_h^\theta]) -  U(t, x, [X_h^\theta])$. Applying the chain rule introduced in \cite{chassagneux2014probabilistic} to $I$ and Ito's formula to $J$, we have
\begin{align*}
I &= \int_0^h \EE \bigg[ \sum_{i=1}^N V_0^i(X_r^\theta, [X_r^\theta], \tilX_r)\partial_\mu U(t, x, [X_r^\theta], X_r^\theta)_i  \\
&\hspace{4em}+ \frac{1}{2}\sum_{i,j=1}^N [\sigma\sigma^\top(X_r^\theta, [X_r^\theta], \tilX_r)]_{i,j}\partial_{v_j}\partial_{\mu}U(t, x, [X_r^\theta], X_r^\theta)_i\bigg] \ud r, \\
J &= \int_0^h \sum_{i=1}^N V_0^i(X_r^{x, [\theta]}, [X^\theta_r], \tilX_r)\partial_{x_i}U(t, X_r^{x, [\theta]}, [X_h^\theta])\ud r \\
& \hspace{2em} + \frac{1}{2}\int_0^h \sum_{i,j=1}^N [\sigma \sigma^\top(X_r^{x, [\theta]}, [X^\theta_r], \tilX_r)]_{i,j} \partial_{x_i}\partial_{x_j} U(t, X_r^{x, [\theta]}, [X_h^\theta])  \ud r \\
&\hspace{2em} + \int_0^h \sum_{j=1}^d\sum_{i=1}^N V_j^i(X_r^{x, [\theta]}, [X^\theta_r], \tilX_r)\partial_{x_i}U(t, X_r^{x, [\theta]}, [X_h^\theta])\ud B_r^j.
\end{align*}

For the meaning of the differential operator with respect to measure $\partial_\mu$ appeared in $I$, we refer to section \ref{sec:Notations and Basic Setup}.
Then let us plug $I, J$ into \eqref{def:diff_U} and take expectation, divide by $h$ on both sides, and send $h$ to $0$, we will end up with a PDE of the form given below
\begin{equation}
    \begin{aligned}
    \label{eq:pde}
   (\partial_t - \mc{L})U(t, x, [\theta])=0 \quad & \text{ for } (t, x, [\theta])\in (0, T]\times \RR^N \times\mc{P}_2(\RR^N),\\ 
   U(0, x, [\theta]) = g(x, [\theta])\quad & \text{ for } (x, [\theta])\in \RR^N \times\mc{P}_2(\RR^N),
\end{aligned}
\end{equation}
where $g: \RR^N \times \mc{P}_2(\RR^N) \to \RR$ and the operator $\mc{L}$ acts on smooth enough functions $F:\RR^N \times \mc{P}_2(\RR^N) \to \RR^N$ defined by
\begin{align}
\label{eq:operator_L}
    \mc{L} F(x, [\theta]) =&\EE\bigg[ \sum_{i=1}^N V_0^i(x, [\theta], \tilde{\theta})\partial_{x_i}F(x, [\theta]) + \frac{1}{2} \sum_{i,j=1}^N [\sigma\sigma^\top(x, [\theta], \tilde{\theta})]_{i,j}\partial_{x_i}\partial_{x_j}F(x, [\theta])\bigg] \nonumber\\
    &\hspace{-4em}+ \EE\bigg[ \sum_{i=1}^N V_0^i(\theta, [\theta], \tilde{\theta})\partial_\mu F(x, [\theta], \theta)_i  + \frac{1}{2}\sum_{i,j=1}^N [\sigma\sigma^\top(\theta, [\theta], \tilde{\theta})]_{i,j}\partial_{v_j}\partial_{\mu}F(x, [\theta], \theta)_i\bigg].
\end{align}
The expectation in the first line of \eqref{eq:operator_L} is taken with respect to the random variable $\tilde{\theta}$ due to the appearance of the neighborhood process in the difference $J$, while the the expectation in the second line is taken with respect to the joint distribution of $\theta, \tilde{\theta} $, as an application of the chain rule introduced in \cite{chassagneux2014probabilistic}  to the difference $I$.

It is evidently that a proper condition for the initial $g$ is needed for the existence of the solution to PDE \eqref{eq:pde}. Such a directed chain type SDE has not been considered before, the closest work is related to the PDE associated with the  McKean-Vlasov type SDE. In \cite{rainer2017}, $g$ is assumed to have bounded second order derivatives. The smoothness on $g$ is relaxed in \cite{crisan2018smoothing}. In particular, they assume $g$ belongs to a class of functions that can be approximated by a sequence of functions with polynomial growth, and also satisfy certain growth condition on its derivatives. Hence, they claim that $g$ is not necessarily differentiable. We shall emphasize that detailed discussion on the choice of assumptions in $g$ is beyond the scope of this paper, but we conjecture that some similar results should also hold for our case and will include this in our future research. 